\documentclass[12pt]{amsart}
\usepackage{amssymb, amscd, amsthm, amsfonts, mathrsfs, tikz-cd}
\usepackage{amsmath}
\usepackage{tikz-cd}

\usepackage{etoolbox}
\makeatletter
\patchcmd{\@settitle}{\MakeUppercase}{\relax}{}{}
\patchcmd{\@setauthors}{\MakeUppercase}{\relax}{}{}
\patchcmd{\@setauthors}{\large}{\large}{}{} 
\makeatother

\usepackage{graphicx}
\usepackage{indentfirst}
\usepackage{thmtools,mathtools}
\usepackage{thm-restate}
\usepackage[colorlinks, bookmarks=true, backref=page]{hyperref}
\hypersetup{
	colorlinks,
	citecolor=blue,
	linkcolor=blue,
	urlcolor=blue}
\usepackage{cleveref}
\usepackage{cite}
\usepackage{verbatim}
\usepackage{chngcntr}

\usepackage{autonum}

\numberwithin{equation}{section}

\newtheorem{theorem}{Theorem}[section]
\newtheorem{cor}[theorem]{Corollary}
\newtheorem{lem}[theorem]{Lemma}
\newtheorem{q}[theorem]{Question}
\newtheorem{prop}[theorem]{Proposition}
\theoremstyle{definition}
\newtheorem{exa}[theorem]{Example}
\newtheorem{defi}[theorem]{Definition}
\newtheorem{rem}[theorem]{Remark}

\DeclareMathOperator{\Isom}{Isom}

\DeclareMathOperator{\sff}{\mathrm{I\!I}}

\DeclareMathOperator{\arccosh}{arcosh}

\newcommand{\gl}{\tn{GL}}
\newcommand{\spl}{\tn{SL}}
\newcommand{\so}{\tn{SO}}
\newcommand{\cg}{\tn{CG}}
\newcommand{\almostgeo}{nearly geodesic }
\newcommand{\nege}{nearly geodesic }
\newcommand{\negen}{nearly geodesic}

\newcommand{\N}{\mathbb{N}}

\newcommand{\ag}{asymptotically geodesic }
\newcommand{\agn}{asymptotically geodesic}

\newcommand{\hms }{Hamenstädt }

\newcommand{\gps}{Gromov--Pyatetski-Shapiro }
\newcommand{\cat}{\tn{CAT}}
\newcommand{\oto}{of type I }
\newcommand{\oton}{of type I}
\newcommand{\nar}{non-arithmetic }

\newcommand{\vl}{Vladimir }
\newcommand{\mkv}{Markovi\'c }

\newcommand{\poi}{$\pi_1$-injective }
\newcommand{\poin}{$\pi_1$-injective}

\newcommand{\cbility}{commensurability }
\newcommand{\kms}{Kahn--Markovi\'c--Smilga }

\newcommand{\catzero}{\tn{CAT(0)} }

\newcommand{\km}{Kahn--Markovi\'c }
\newcommand{\totallygeodesic}{totally geodesic }

\newcommand{\area}{\textnormal{Area}}

\newcommand{\inj}{\textnormal{inj}}

\newcommand{\neta}{N_{\eta}}

\newcommand{\ch}{\mathcal{CH}}
\newcommand{\gnmo}{\mathcal{G}_{n}}

\newcommand{\SONO}{\tn{SO}(n,1)^{\circ}}
\newcommand{\SOPO}{\tn{SO}(n+1,1)^{\circ}}

\newcommand{\mhat}{\hat{M}}

\newcommand{\tx}{\tilde{x}}
\newcommand{\tpj}{\widetilde{p_j}}

\newcommand{\tpaij}{\widetilde{\Pi_j}}

\newcommand{\txj}{\widetilde{x_j}}
\newcommand{\tpai}{\widetilde{\Pi}}

\newcommand{\qf}{quasi-Fuchsian }

\newcommand{\gt}{\mathcal{G}_2}
\newcommand{\gtm}{\mathcal{G}_2M}

\newcommand{\tgamma}{\tilde{\gamma}}

\newcommand{\tsj}{\tilde{S}_j}
\newcommand{\tsi}{\tilde{S}_i}
\newcommand{\ts}{\widetilde{S}}
\newcommand{\ws}{\widetilde{S}}
\newcommand{\psltr}{PSL_2(\R)}

\newcommand{\gm}{\mathcal{G}_{n}}

\newcommand{\tm}{\widetilde{M}}

\newcommand{\hy}{\mathbb{H}^3}

\newcommand{\image}{\textnormal{Im}}

\newcommand{\Q}{\mathbb{Q}}
\newcommand{\R}{\mathbb{R}}

\newcommand{\tn}{\textnormal}

\newcommand{\T}{\mathbb{T}}

\newcommand{\vol}{\textnormal{Vol}}
\newcommand{\Z}{\mathbb{Z}}

\newcommand{\lifs}{L^{\infty}(S)}

\newcommand{\pr}{\tn{Pr}}

\newcommand{\hs}{\mathbb{H}^{n+1}}

\newcommand{\im}{\tn{Im}}

\newcommand{\dist}{\textbf{\tn{d}}}

\usepackage{lipsum}
\let\OLDthebibliography\thebibliography
\renewcommand\thebibliography[1]{
	\OLDthebibliography{#1}
	\setlength{\parskip}{1.5pt}
	\setlength{\itemsep}{0pt plus 0.25ex}
}

\title[\large{Asymptotically geodesic hypersurfaces and $\pi_1$ of hyperbolic manifolds}]{\Large Asymptotically geodesic hypersurfaces and the fundamental groups of hyperbolic manifolds}

\author{\large Xiaolong Hans Han}

\address{Center for Mathematics and Interdisciplinary Sciences, Fudan University, Shanghai, 200433, China\newline\indent
Shanghai Institute for Mathematics and Interdisciplinary Sciences (SIMIS), Shanghai, 200433, China}
\email{xhh@simis.cn}
	\author{\large Ruojing Jiang}

\address{Massachusetts Institute of Technology, Department of Mathematics, Cambridge, MA 02139}
\email{ruojingj@mit.edu}

\date{\today}
\begin{document}

\begingroup
\def\uppercasenonmath#1{} 
\maketitle
\endgroup

\begin{abstract}
 We consider closed hypersurfaces smoothly immersed in hyperbolic manifolds up to homotopy and commensurability. 
	We prove that if a closed hyperbolic manifold $M$ contains a sequence of \emph{\agn} hypersurfaces, then $\pi_1(M)$ is virtually special and hence embeds in $\spl_d(\Z)$ for some $d\in \N$. If $M^{n+1}$ ($n\geq 2$) is, in addition, arithmetic of type I, we constructs a sequence 
    of hypersurfaces  which are \ag (but not  totally geodesic), strongly filling, and equidistributing in the $n$-Grassmann bundle over $M$. This partially answers a question of Al Assal--Lowe. As a corollary, for each cocompact arithmetic lattice $\Gamma$ of $\SOPO$ of type I, there exist infinitely many arithmetic and infinitely many non-arithmetic cocompact lattices $H$ of $\SONO$ that admit monomorphisms $\iota_*\colon H\to \Gamma$ which do not extend to a Lie group homomorphism from  $\SONO$ into $\SOPO$.
\end{abstract}

\section{Introduction}
Let $M$ be a closed orientable hyperbolic manifold. All hypersurfaces $S\hookrightarrow M$ are assumed to be closed, orientable, and smoothly immersed. 
We consider hypersurfaces up to homotopy and commensurability.
Rubinstein--Sageev \cite{rsFillingSurfaces} define a hypersurface $S\hookrightarrow M$ to be \emph{strongly filling} if for any $p\neq q \in \partial_{\infty} \hs$ the sphere at infinity, there exists a lift $\tilde{S} \subset \hs$ of $S$ such that $p,q$ belong to different components of $\partial_{\infty} \hs - \partial_{\infty}\tilde{S}$. If $S$ is strongly filling, $S$ is filling. A sequence of hypersurfaces $S_i\hookrightarrow M$ is \emph{asymptotically geodesic } if their second fundamental forms $\sff_{S_i}$ satisfy $\Vert \sff_{S_i}\Vert_{L^\infty(S_i)} \to 0$. 
The superscript $G^\circ$ denote the identity component of a group. Al Assal--Lowe [\citenum{alAsymptoticallyGeodesicsurfaces}, Question 2] ask whether hyperbolic manifolds of dimension $\geq 4$ contains asymptotically geodesic hypersurfaces that are not totally geodesic. 

\begin{theorem}\label{arithmetic of type I contains nice hypersurfaces}
Let $\Gamma$ be a cocompact arithmetic lattice of $\SOPO$ of type I and $M= \Gamma\backslash \hs$. Then there exists a sequence $S_i$ of closed hyperbolic manifolds of dimension $n$ which admit \poi immersions $\iota\colon S_i \hookrightarrow M$ such that 
\begin{enumerate}
	\item $\iota(S_i)$ is strongly filling, asymptotically geodesic, and not homotopic to totally geodesic hypersurfaces; 
	\item $\iota(S_i)$ is equidistributing in the Grassmann bundle $\gm M$. 
\end{enumerate}
Moreover, we can take $\pi_1(S_i)$ as arithmetic or non-arithmetic lattices of $\SONO$. 
\end{theorem}

There are infinitely many arithmetic lattices of $\so(2n+1,1)^{\circ}$ \oton; an arithmetic lattice of $\so(2n,1)^{\circ}$ must be \oto (see [\citenum{vsDiscreteGroupsConstantCurvature}, p.221]). Results in \Cref{arithmetic of type I contains nice hypersurfaces} also generalize to non-cocompact arithmetic manifolds which contain a closed totally geodesic hypersurface (such manifolds only exist up to dimension $4$ by Meyer's theorem on quadratic forms).  
We are grateful for Ben Lowe's suggestion for the following. 
\begin{cor}\label{non superrigid}
	Let  $\Gamma$ be a cocompact arithmetic lattice of $\SOPO$ of type I (with $n\geq 2$). Then there exist infinitely many arithmetic and infinitely many non-arithmetic cocompact lattices $H$ of $\SONO$ that admit injective homomorphisms $\iota_*\colon H\to \Gamma$ which do not extend to a Lie group homomorphism of $\SONO$ into $\SOPO$. 
\end{cor}
Delzant--Gromov \cite{dgCutKahler} proves that complex hyperbolic lattices do not contain quasi-convex codimension-$1$ subgroups; \cite{ckSuperrigidity}  establishes the superrigidity of quaternionic and Cayley hyperbolic lattices which forbids non totally geodesic immersion of hypersurfaces as in \Cref{non superrigid}. By Margulis' Superrigidity Theorem \cite{mgDiscreteSubgroupsNonpositive}, if $\Gamma$ is an irreducible lattice in a connected, center-free, higher--rank semisimple Lie group $G$ with no compact factors, then any homomorphism $\rho\colon \Gamma\to GL_n(\mathbb{R})$ with unbounded, Zariski--dense image virtually extends to a continuous representation of $G$.

 Haglund--Wise [\citenum{hwSpecialCubeComplex}] pioneers the study of \emph{special} cube complexes and define a group $G$ to be \emph{special} if $G$ acts freely and cocompactly on a CAT(0) cube complex $X$, such that the quotient $G\backslash X$ is special. [\citenum{hwSpecialCubeComplex}, Theorem 1.1] points out that a special group $G$ embeds in $\tn{SL}_d(\Z)$ for some $d \in \N$. 
A group $G$ (resp. manifold $M$) \emph{virtually} satisfies a property $P$ if there is a finite-index subgroup $H$ of $G$ (resp. finite cover $\hat{M}$ of $M$) which satisfies $P$. A submanifold $S\hookrightarrow M$ homologically injects if $H_k(S;\Z)\rightarrow H_k(M;\Z)$ is injective for all integer $k\geq 0$.
Our second theorem is the following. 
\begin{theorem}\label{asymptotically geodesic implies virtually special}
	If a closed hyperbolic manifold $M$ contains a sequence of \ag hypersurfaces, then $\pi_1(M)$ acts properly and cocompactly on a \catzero cube complex with one family of hyperplanes, and $\pi_1(M)$ is virtually special. Moreover, every closed hypersurface $S$ with $\sff_{[S]}<1$ is virtually embedded and virtually homologically injects. 
\end{theorem}
We denote by $[S]$ the class of hypersurfaces homotopic and commensurable to $S$ in $M$. 
\begin{defi}[The principal curvature spectrum]\label{principal curvature spectrum}
	Let $\Lambda_M$ be the functional defined on the set of classes of closed, smoothly immersed hypersurfaces in $M$ by
	\begin{equation}
	\Lambda_M\colon	[S]\mapsto   \inf\limits_{S'\in[S]} \|\sff_{S'}\|_{L^{\infty}}\eqqcolon \sff_{[S]}.
	\end{equation}
We call the image $\im(\Lambda_M)$ \emph{the principal curvature spectrum} of $M$. 
\end{defi}
\Cref{principal curvature spectrum} is implicit in Al Assal-Lowe \cite{alAsymptoticallyGeodesicsurfaces}, and its Question 2 asks about whether there exists a gap at $0$ of $\im(\Lambda_M)$. 

By Thurston's observation [\citenum{lcSmallCurvatureSurface}, Theorem 5.1], if $\sff_{[S]}<1$, then every representative of $[S]$ is \poin. \Cref{zeros of spectrum} proves that if $\sff_{[S]}=0$, then $[S]$ contains a totally geodesic representative. 
 \begin{cor}\label{spectrum and strongly filling}
 	For any closed hyperbolic manifold $M$, there exists a constant $C_M< 1$  such that if $[S]$ satisfies $0<\sff_{[S]}<C_M$, then $[S]$ is strongly filling. Thus if
 	$\image(\Lambda_M) \cap (0,C_M]$ is non-empty, the conclusions of \Cref{asymptotically geodesic implies virtually special} hold. 
 \end{cor}
If $M$ is a closed hyperbolic $3$-manifold (resp. compact arithmetic of type I), it follows from \cite{kmImmersingAlmostGeodesic,saMinimalDiscsHyperbolicSpace} (resp. \Cref{arithmetic of type I contains nice hypersurfaces}) that $0$ is an accumulation point of $\image(\Lambda_M).$  
	
In establishing the $\Z$-linearity of $\pi_1(M)$ in \Cref{asymptotically geodesic implies virtually special}, it becomes clear that some of the basic group theoretic properties hold not only for \ag hypersurfaces but also for \emph{\nege} hypersurfaces, i.e., $S \hookrightarrow M$ such that $\|\sff_S\|_{\lifs}<1$. Building on [\citenum{esGaussMapEquivariantImmersion}, Section 4], we prove that for a closed \nege hypersurface $S\hookrightarrow M$, $\pi_1(S)\rightarrow \pi_1(M)$ is injective, quasi-convex, and codimension-$1$, in \Cref{nearly geodesic hypersurfaces}. 
This generalizes \emph{nearly Fuchsian} surfaces in hyperbolic $3$-manifolds, initiated by Uhlenbeck \cite{ukclosedminSurfaceHyperbolic}, for whose many exciting developments we refer to \cite{hlsUniqueness} and \cite{esGaussMapEquivariantImmersion} and the references therein. 

\subsection{Outlines of proofs}
An ingredient of the proof of \Cref{arithmetic of type I contains nice hypersurfaces} is the following. The proof boils down to the polynomial divergence of unipotent flows in $\SONO$ (Ratner, Shah  \cite{rmRatner'stheorem,snRatnerTheorem} and Mozes--Shah \cite{msErgodicInvariantMeasure}).  
\begin{prop}\label{1.2}
	Let $N$ be a finite-volume hyperbolic $(n+1)$-manifold with $n\geq 2$. Then all but finitely many pairs of closed totally geodesic hypersurfaces $\Sigma_1, \Sigma_2$ pairwise intersect into a codimension-$2$ submanifold which is strongly filling in $\Sigma_1$ or $\Sigma_2$. 
\end{prop}
By Bader-Fisher--Miller--Stover \cite{bfmsTotallyGeodesicSubmanifolds} (for $n\geq 2$) and Margulis--Mohammadi \cite{mmArithmeticityHyperbolic} (for $n=2$), if a hyperbolic manifold $M^{n+1}$ contains infinitely totally geodesic hypersurfaces, then $M$ is arithmetic. 

The proof of \Cref{arithmetic of type I contains nice hypersurfaces} starts by showing that we can take a pair of intersecting totally geodesic hypersurfaces $\Sigma_1$ and $\Sigma_2$ (or a single self-intersecting hypersurface) which contains an intersection locus whose angle is close to $\pi$. By lifting $\Sigma_1$ and $\Sigma_2$ to a finite cover $\hat{M}$, we obtain embedded $\hat{\Sigma}_1$ and $\hat{\Sigma}_2$ with intersection angle around $\pi$ along $L=\hat{\Sigma}_1\cap\hat{\Sigma}_2$ (by [\citenum{bhwHyperplaneSection}, Corollary 1.6]). We cut $\hat{\Sigma}_1$ and $\hat{\Sigma}_2$ along $L$ and reglue them to form a single pleated hypersurface $\hat{S}_0$ with small bending angle. 
We then smooth $\hat{S}_0$ along $L$ to obtain a closed hypersurface $ S$ immersed in $M$. As the angle between $\hat{\Sigma}_1$ and $\hat{\Sigma}_2$ tends to $\pi$, the resulting hypersurfaces have principal curvatures tending to $0$, and hence form a sequence of asymptotically geodesic hypersurfaces. Moreover, we show that these hypersurfaces equidistribute in the Grassmann bundle of $M$. This follows from the equidistribution of totally geodesic hypersurfaces together with the fact that the smoothing procedure produces hypersurfaces whose tangent planes remain close to those of the totally geodesic hypersurfaces. To establish that the hypersurfaces constructed are not homotopic to totally geodesic ones, \Cref{totally geodesic area-minimizing} proves that a totally geodesic immersion $S\hookrightarrow M$ of closed hyperbolic manifold $S$ is area-minimizing in its homotopy class, but our smoothing procedure strictly decreases the area of the pleated hypersurface as in [\citenum{kmsGeometricallyTopologicallyRandomsurface}, Figure 2]. 

An outline of the proof of \Cref{asymptotically geodesic implies virtually special} is the following. If $S_i$ are asymptotically geodesic hypersurfaces, lifting $S_i$ to $\tsi$ in $\hs$ and using $k$-cones centered at $(\tilde{x}_i,\tilde{\Pi}_i)\in \gm \tsi$ (which trap $\tilde{S}_i$ in its interior), we prove that $S_i$ must be asymptotically dense in $\gm M$. 

Consequently, 
 all but finitely many $S_i$ are strongly filling (proved in Proposition~\ref{ag implies eventually strongly filling}). Then the conjugates of a single $\{\pi_1(S_i)\}$ separates any pair of distinct boundary points of $\pi_1(M)$. This yields a proper cocompact action of $\pi_1(M)$ on a CAT(0) cube complex by [\citenum{bwBoundaryCriterionCubulation}, Theorem 1.4], and $\pi_1(M)$ is therefore virtually special by \cite{aiVirtualHakenConj}. 
Thus $\pi_1(M)$ embeds into \emph{right-angled Artin groups (RAAGs)} (a right-angled Artin group is a group presented by generators corresponding to the vertices of a finite graph, with commutation relations only along the edges). Hence, $\pi_1(M)$ is linear over $\mathbb Z$. By adapting ideas in the proof of [\citenum{bhwHyperplaneSection}, Theorem 1.2], all quasi-convex subgroups are virtual retract and hence virtually homologically inject. 

	\subsection{Some histories and motivations}\label{history}
	We give a brief account on results on \ag hypersurfaces, superrigidity, and immersions of hyperbolic manifolds that are relevant to the paper. The proofs of the main theorems rely crucially on the works on special cube complex and virtual retracts in Sageev \cite{smCubeComplex,smCodim1Subgroups}, Haglund-Wise \cite{hwSpecialCubeComplex}, Bergeron-Wise \cite{bwBoundaryCriterionCubulation}, Bergeron-Haglund-Wise \cite{bhwHyperplaneSection}, and Agol \cite{aiVirtualHakenConj}, and we provide much more details in \Cref{from_asy_geod_to_cube_cpl}. 
	
	Leininger  [\citenum{lcSmallCurvatureSurface}, Theorem 4.1] constructs a sequence of \ag surfaces $S_i$ \emph{embedded} in different manifolds $M_i$. Fix a closed hyperbolic $3$-manifold $M^3$. The first author \cite{hxhNearlyGeodesicFilling} proves that \ag surfaces $S_i$ in  $M^3$ are eventually strongly filling and thus must be self-intersecting; Al Assal--Lowe \cite{alAsymptoticallyGeodesicsurfaces} and \cite{hxhNearlyGeodesicFilling} prove that \ag minimal $S_i \hookrightarrow M^3$ are asymptotically dense $\gt S_i \to \gt M^3$. For any $\epsilon>0$, \km \cite{kmImmersingAlmostGeodesic} construct infinitely many $(1+\epsilon)$-\qf surfaces in $M^3$. Building on \cite{kmImmersingAlmostGeodesic} and \hms  \cite{huIncompressibleSurfacesRankOne}, Calegari-Marques-Neves  [\citenum{cmnCountingminimal}, Theorem 4.2] construct in $M^3$ a sequence of asymptotically geodesic surfaces whose limiting measure has a nontrivial Lebesgue component. Labourie \cite{flAsymptoticCounting} shows we can find such surfaces (not necessarily connected) that equidistribute.
	Lowe-Neves \cite{lnMinSurEntropy} show that the surfaces can be taken to be connected.  \kms \cite{kmsGeometricallyTopologicallyRandomsurface} show that sequences of geometrically random asymptotically geodesic surfaces have a nonzero Lebesgue component for any weak-* limit. 
	Al Assal \cite{afLimitsAsymptoticalFuchsian} proves that the limiting measures induced by \ag surfaces in $M^3$ is exactly the space of $\psltr$-invariant measures on $\gtm^3$. Filip--Fisher--Lowe \cite{fflFiniteTotalGeo} proves that a closed negatively curved analytic Riemannian manifold which contains infinitely many totally geodesic hypersurfaces are arithmetic. 
	
	 Let \(S = H\backslash G/K\), \(M = \Gamma \backslash G'/K'\) be symmetric spaces, with \(G,G'\) semisimple Lie groups, center–free, no compact factors, where \(K \leq  G\), \(K'\leq G'\) are maximal compact subgroups, and $H$ and $\Gamma$ are irreducible lattices in $G$ and $G'$, respectively. Margulis superrigidity theorem implies that, if, in addition, $G$ is not isogenous to any group that is of the form $\SONO \times  K$ or $ \tn{SU} (m,1) \times  K$ ($K$ is compact), then for any homomorphism $\phi\colon H  \to  G'$ with $\phi (H)$ Zariski dense in $G'$, $\phi$  extends to a continuous homomorphism $\phi  \colon  G \to  G'$. If there is a continuous embedding $\alpha\colon G \to G'$ extending an injective homomorphism $\iota_*\colon H \to \Gamma$, as outlined in [\citenum{mdArithmeticGroups}, 16.2(iii)], then $\iota\colon S\to M$ corresponds to totally geodesic immersion (more detail is provided in \Cref{proof of 1.2}). This generalizes Mostow's rigidity theorem where $\iota_*$ is an isomorphism and $\iota$ is an isometry. 
	  Margulis arithmeticity theorem \cite{mgSuperrigidity1,mgSuperrigidity2} proves that an irreducible lattice in a semisimple Lie group of real rank at least $2$ is always arithmetic. 

	Corlette \cite{ckSuperrigidity} extends the superrigidity to lattices in quaternionic and Cayley hyperbolic spaces. Johnson--Millson [\citenum{jmDeformationHyperbolic}, 3] show that hyperbolic lattices $H$ in $\SONO$ are not necessarily superrigid in $\SOPO$ by constructing families of deformations.  \Cref{non superrigid} says that superrigidity of hyperbolic lattices $H <\SONO$ in $\SOPO$ can fail for representations who image is in addition required to be contained in a lattice in $\SOPO$, even for representations that are arbitrarily close to being totally geodesic. Delzant--Gromov in [\citenum{dgCutKahler}, Corollary] prove that complex hyperbolic lattices (that are not commensurable to surface subgroups) cannot have codimension-$1$ quasiconvex subgroups. 
	
	Lackenby \cite{lmSurfaceSubgroupsTor} proves that any finitely generated, Kleinian group that contains a finite, non-cyclic subgroup either is finite or virtually free or contains a surface subgroup. This implies that arithmetic Kleinian group contains a surface subgroup. 
	Bergeron--Haglund--Wise [\citenum{bhwHyperplaneSection}, Proposition 9.1] show that  a \gps non-arithmetic cocompact lattice in $\SONO$ virtually embeds as a quasiconvex subgroup of a standard cocompact arithmetic lattice in $\SOPO$. Kolpakov--Reid--Slavich \cite{krsEmbedArithmetic} proves that any arithmetic hyperbolic $n$-manifold of type I can either be embedded into an arithmetic hyperbolic $(n + 1)$-manifold as a totally geodesic submanifold or its universal mod $2$ abelian cover can. 
	Kolpakov--Riolo--Slavich \cite{krsEmbedNonarithmetic} shows that many Gromov--Pyatetski-Shapiro and Agol--Belolipetsky--Thomson inter-bred \nar hyperbolic manifolds embed into higher-dimensional hyperbolic manifolds as codimension-one totally geodesic submanifolds. 

	\Cref{asymptotically geodesic implies virtually special} gives a new proof of [\citenum{bwBoundaryCriterionCubulation}, Theorem 6.2], since arithmetic manifolds \oto contains a sequence of totally geodesic hypersurfaces. Moreover, in the proof of [\citenum{bwBoundaryCriterionCubulation},  Theorem 6.2] which use commensurators, there is no uniform, effective upper bound on the number of subgroups needed to fulfill [\citenum{bwBoundaryCriterionCubulation}, Theorem 1.4], since the proof relies on the compactness of the triple product of $\partial \pi_1(M)$. 
	
	We refer to the introduction of \cite{hxhNearlyGeodesicFilling} for more motivations and history on filling hypersurface and rigidity of unipotent flows. 
	
\textbf{Outline. }	
In \Cref{preliminaries}, we review arithmetic manifolds and totally geodesic hypersurfaces. We then provide some basic backgrounds on cube complex, virtual retracts, strongly filling subgroups, and virtual specialness in  \Cref{from_asy_geod_to_cube_cpl}.
In \Cref{ggt}, we define \nege hypersurfaces in hyperbolic manifolds, generalizing almost Fuchsian surfaces. We also define \ag hypersurfaces, and prove that in a closed hyperbolic manifold, they are asymptotically dense and eventually strongly filling \Cref{ag implies eventually strongly filling}. This establishes \Cref{asymptotically geodesic implies virtually special} and \Cref{spectrum and strongly filling} with the results in \Cref{from_asy_geod_to_cube_cpl}. 
In \Cref{constructions}, we construct non-arithmetic \ag hypersurfaces in arithmetic manifolds of type I, and establish \Cref{arithmetic of type I contains nice hypersurfaces}. 

\tableofcontents

\section{Preliminaries}\label{preliminaries}
If $M$ is a smooth manifold, the Grassmann $d$-plane bundle $\mathcal{G}_dM$ is the set $\{(x, \Pi)|x\in M, \Pi \subset T_xM \tn{ is a }d\tn{-dimensional subspace}\}$, equipped with the natural topology induced from the tangent bundle $TM$. If $M$ has $(n+1)$-dimension, then there is a double covering map from the unit tangent bundle $T^1M $ to $\gm M$. If $M$ is a Riemannian manifold, there is a natural metric on $\gm M$ compatible with the Riemannian metric on $M$, which measures the proximity of the positions and directions of two elements in $\gm M$. For more detail, see e.g., [\citenum{hxhNearlyGeodesicFilling}, 2.2]. 

\subsection{Arithmetic manifolds}
References for arithmetic manifolds include e.g. [\citenum{vsDiscreteGroupsConstantCurvature}, Chapter 6] and \cite{mdArithmeticGroups}. 
Let $K\subset \R$ be a totally real algebraic number field, and $R_K$ the ring of its integers. A non-degenerate quadratic form
$$f(x)=\sum_{i,j=0}^{n}a_{ij}x_ix_j (a_{ij}=a_{ji}\in K)$$
is \emph{admissible} if its negative index is $1$, and for any non-identity embedding $\sigma\colon K \rightarrow \R$, the quadratic form 
$$f^{\sigma}(x)=\sum_{i,j=0}^{n} a_{ij}^{\sigma}x_ix_j$$
is positive definite. Then the group $O'(f,R_K)$ of linear transformations with coefficients in $R_K$ preserving the form $f$ and mapping each connected component of the cone $C=\{x\in \R^{n+1}: f(x)<0\}$ onto itself is a discrete group of isometries of $\mathbb{H}^n$. According to the general theory of arithmetic discrete groups, if $K=\Q$, then $O'(f,R_K)\backslash\mathbb{H}^n$ is finite-volume non-compact; otherwise it is compact. 

Instead of considering the standard lattice $R_K^{n+1}\subset K^{n+1}$, if we consider linear transformations preserving an arbitrary lattice $L\subset K^{n+1}$, the resulting group is commensurable with $O'(f,R_K)$. Considering the lattice $L$ as a quadratic $R_K$-module, with scalar product defined by the form $f$, we denote this group by $O'(L)$. We call $O'(L)$ an \emph{arithmetic lattice \oton}, and $M=O'(L)\backslash \mathbb{H}^n$ an \emph{arithmetic manifold \oton}. 

It is well-known to the experts that an arithmetic real hyperbolic manifold which contains a totally geodesic hypersurface is \oto (following from Tits's classification of algebraic groups, see [\citenum{bhwHyperplaneSection}, 1.1]). 

\begin{lem}[\citenum{vsDiscreteGroupsConstantCurvature}, p.221]
	All arithmetic lattices of $\so(2n,1)$ are \oton. 
\end{lem}

\begin{exa}[Infinitely many cocompact arithmetic lattices of $\so(n,1)$ \oton]
	Let $d$ be a squarefree positive integer. Consider the admissible quadratic form
	$$f(x)=x_0^2+x_1^2\cdots +x_{n-1}^2-\sqrt{d}x_{n}^2$$
	over the totally real number field $K=\Q(\sqrt{d})$ of degree $2$. Then by [\citenum{awIntroAlgebraicNumberTheory}, Theorem 5.4.2], the ring of integers is
	$$R_K= \begin{cases}
		\Z+\Z\sqrt{d} & \text{if } d \not\equiv 1 (\tn{mod } 4),\\
		\Z+\Z\frac{1+\sqrt{d}}{2}  & \text{if } d \equiv 1 (\tn{mod } 4).
	\end{cases}$$
	Hence, the orthogonal group
	 $$\tn{O}^+ (f, R_K)\coloneqq \{T\in \gl(n+1, R_K): f(Tx)=f(x) \quad\forall x\in \R^n\}$$ 
	has a torsion-free subgroup $H$ of finite index by [\citenum{rjFoundationsHyperbolicManifolds}, Theorem 7.6.7]. Let $M$ be the diagonal matrix with $n+1$ diagonal entries $1, \cdots, 1, d^{-1/4}$. Then $f(Mx)=x_0^2+x_1^2\cdots +x_{n-1}^2-x_{n}^2$ the standard one for all $x\in \R^{n+1}$. Let $G =M^{-1}HM$. Then the space-form $G \backslash \mathbb{H}^n $ is compact for each $n>1$ by [\citenum{rjFoundationsHyperbolicManifolds}, Theorem 12.8.8]. Therefore, there are infinitely compact arithmetic hyperbolic manifolds of type I of dimension $n>1$. 
	
	Given a totally real field $K$, two admissible, $K$-defined quadratic forms $f_1$ and $f_2$ define the same commensurability class of arithmetic hyperbolic lattices if and only if $f$ is equivalent over $K$ to $\lambda\cdot f_2$ for some $\lambda\in K-\{0\}$
(see [\citenum{bbksSubspaceHyperbolicLattices}, Remark 3.2]).
\end{exa}
More sophisticated examples using Salem numbers are e.g. Example 8 of [\citenum{rjFoundationsHyperbolicManifolds}, 12.8]. 
\begin{rem}
	One can construct arithmetic lattices of $\so(2n+1, 1)$ of type II using quaternion algebras, e.g., [\citenum{mdArithmeticGroups}, (6.4.8) Proposition]; [\citenum{mdArithmeticGroups}, (6.4.12) Proposition] then says for $n \notin \{3, 7\}$ type I and II exhaust all the cocompact, arithmetic subgroups of $\so(n,1)$ (up to commensurability and conjugates). 
\end{rem}

\subsection{Tubular neighborhood of totally geodesic hypersurfaces}
Denote the volume of the $(n+1)$-dimensional hyperbolic ball of radius $r$ by, $V_{n+1}(r)$. Let $r\colon \R^+ \rightarrow \R^+$ be the function
$$r(x)=\log\coth\frac{x}{2},$$ 
which monotonically decreases and satisfies $r\circ r = \tn{Id}_{\R^+}$. 

Since $V_{n}(x)$ is an increasing function of $x$, the composition
$(V_{n}\circ r)(x)$ is decreasing and hence has an inverse. 
Define the $(n+1)$-dimensional \emph{tubular neighborhood function} to be
$$c_{n+1}(A)=\frac{1}{2}(V_{n}\circ r)^{-1}(A);$$
$c_{n+1}(A)$ is monotone decreasing in $A$ and goes to zero as $A$ goes to infinity. 

\begin{theorem}[\citenum{baTubularNeighborhoodsTotallyGeodesic}, Theorem 1.1]
	Suppose $M^{n+1}$ is a hyperbolic manifold containing $\Sigma$, an embedded
	closed totally geodesic hypersurface of area $A$. Then $\Sigma$ has a tubular neighborhood of width $c_n(A)$. That is, the set of points
	$$\{x\in M: d(x, \Sigma)<c_n(A)\}$$
	is isometric to the warped product $(-c_n(A),c_n(A)) \times_{\cosh}\Sigma$. 
	Furthermore, any disjoint set of such hypersurfaces have disjoint tubular neighborhoods.
\end{theorem}

\begin{lem}[\citenum{baTubularNeighborhoodsTotallyGeodesic}, p. 214]
	The hyperbolic metric restricted to the hypersurface of a constant distance $t$ from
	$\mathbb{H}^{n}$ is (by using normal coordinates), 
	\begin{equation}
			\cosh^2(t)\langle \cdot ,\cdot \rangle
	\end{equation}
	where the above inner product $\langle \cdot,\cdot \rangle$ is the $n $-dimensional hyperbolic metric on the fixed copy of $\mathbb{H}^{n}$. 
\end{lem}

[\citenum{hxhNearlyGeodesicFilling}, 2.1] provides some backgrounds on quasi-Fuchsian and minimal surfaces of hyperbolic $3$-manifolds.

\section{Cube complex, virtual retract, and strong fillingness}\label{from_asy_geod_to_cube_cpl}
\subsection{Basics}
We start by recalling some backgrounds on cube complex. 

A \emph{cube complex} is a metric polyhedral complex all of whose cells are unit cubes, i.e., it is the quotient of a disjoint union of copies of unit cubes under an equivalence relation generated by a set of isometric identifications of faces of cubes. 
Let $X$ be a CAT(0) cube complex. We define an equivalence relation on the edges of $X$ generated by identifying opposite edges of some square in $X$. Given an equivalence class $[e]$ of edges, the hyperplane $w$ dual to $[e]$ is the collection of midcubes which intersect edges in $[e]$. The complement $X- w$ consists of two halfspaces. 

Let $G$ be a finitely generated group with a Cayley graph $\Delta$. A subgroup $H\leq G$ is \emph{quasi-convex} if for any $h_1, h_2 \in H$, any geodesic $[h_1,h_2]$ in the Cayley graph of $G$ lies within $k$-neighborhood of $H$. 
	A subgroup $H\leq G$ is \textit{codimension}-$1$ if it has a finite neighborhood $N_r(H)$ such that $\Delta-N_r(H)$ contains at least two $H$-orbits of \textit{deep} components, which do not lie in any $N_s(H)$. In \cite{smCodim1Subgroups}, Sageev commented that an immersed, orientable, incompressible surface in a closed, orientable $3$-manifold corresponds to a codimension-$1$ surface subgroup of its fundamental group. Surface subgroups of $3$-manifolds are well-studied and we lack a unifying picture for hypersurface subgroups of higher-dimensional  manifolds. However, with the same reasoning as Sageev, we conclude that if $M$ is a closed orientable negatively curved manifold and $S$ a closed, orientable \poi hypersurface with $\pi_1(S)$ a quasi-convex subgroup of $\pi_1(M)$, then $\pi_1(S)$ is a codimension-$1$ subgroup of $\pi_1(M)$. This happens because $\pi_1(M)$ is quasi-isometric to the universal cover $\tilde{M}$ and quasi-convexity of $\pi_1(S)$ ensures that $\pi_1(S)\backslash\tilde{M}$ has at least two ends. 
	 
	If $M$ is a closed hyperbolic manifold, $\pi_1(M)$ is a word-hyperbolic group. Since $\pi_1(M)$ is quasi-isometric to the universal cover $\hs$, [\citenum{bhMetricSpaceNonpositive}, 3.9 Theorem] says $\partial \pi_1(M)$ is homeomorphic to $\partial \hs$. 

If $G$ is a finitely generated group with a finite collection of codimension-$1$ subgroups $H_1,\cdots, H_k$, Sageev \cite{smCubeComplex} introduced a powerful construction that produces an action of $G$ on a CAT(0) cube complex $X$ that is dual to a system of walls corresponding to these subgroups. This elucidates many structures of the group. Moreover, he established the following useful criterion for compactness property. 
\begin{theorem}[Sageev, \citenum{smCodim1Subgroups}]\label{cube complex cocompact}
	Let $G$ be a word-hyperbolic group, and $H_1,...,H_k$ be a  collection of quasiconvex codimension-$1$ subgroups. Then the action of $G$ on the 	dual cube complex is cocompact.
\end{theorem} 
The important properness criterion for the action on cube complex is the following. 
\begin{theorem}[\citenum{bwBoundaryCriterionCubulation}, Theorem 1.4]\label{bergeron-wise}
	Let $G$ be word-hyperbolic. Suppose for each pair of distinct points $(u,v) \in (\partial G)^2$ there exists a quasiconvex codimension-$1$ subgroup $H$ such that $u$ and $v$ lie in the $H$-distinct components of $\partial G-\partial H$. 
	
	Then there is a finite collection $\{H_1, \cdots, H_k\}$ of quasiconvex codimension-$1$ subgroups such that $G$ acts properly and cocompactly on the dual CAT(0) cube complex. 
\end{theorem}
The number $k$ in \Cref{bergeron-wise} depends on the compactness of the triple product $(\partial G)^3\coloneqq \{(u,v,w): u,v,w \in \partial G, u\neq v\neq w\neq u\}$ and thus there is no general method estimating $k$. 
\begin{defi}[Strongly filling subgroups]
	Let $G$ be a word-hyperbolic group and $H$ a quasiconvex, codimension-$1$ subgroup. We call $H$ a \emph{strongly filling subgroup} if for any $p,q \in \partial G$, there exists a $g\in G$ such that $p, q$ lie in $H$-distinct components of $\partial G-\partial gHg^{-1}$. 
\end{defi}
A \nege  hypersurface $S\hookrightarrow M$ is one such that $\|\sff_S\|_{L^\infty(S)} <1$. 
\Cref{nearly geodesic hypersurfaces} proves that $\pi_1(S)\rightarrow \pi_1(M)$ is injective, quasiconvex, and codimension-$1$ subgroup. 
\begin{cor}\label{strongly filling hypersurface give sf groups}
	Let $S$ be a \negen, strongly filling hypersurface of a closed hyperbolic manifold $M$. Then $\pi_1(S)$ is a strongly filling subgroup of $\pi_1(M)$. 
\end{cor}

\begin{theorem}[\citenum{aiVirtualHakenConj}, Theorem 1.1]\label{virtual haken conjecture}
	If $G$ is a word-hyperbolic group acting properly and cocompactly on a CAT(0) cube complex, then $G$ is virtually special. 
\end{theorem}
The proof of \Cref{virtual haken conjecture} relies crucially on the work of \cite{kmImmersingAlmostGeodesic,smCubeComplex,bwBoundaryCriterionCubulation}. Building on \cite{hwSpecialCubeComplex,aiCriteriaVirtualFibering}, Agol \cite{aiVirtualHakenConj} resolves the virtual fibered conjecture of Thurston.  
\subsection{Virtual retracts and homological injections}
We start by recalling that a retraction of a group implies injections of homology.
Let \(G\) be a group and \(H\le G\) a subgroup. A \emph{retraction} is a homomorphism
$
r \colon  G \to H
$
such that \(r|_H = \mathrm{id}_H\). Equivalently, with the inclusion \(i\colon H\hookrightarrow G\), we have
$
r\circ i = \mathrm{id}_H.
$
\begin{lem}\label{retraction implies injection}
	If $H$ is a subgroup of $G$ which is a retraction, then we have injections
	$$H_k (H;\mathbb{Z}) \hookrightarrow H_k (G;\mathbb{Z}) \tn{ for all } k.$$
\end{lem}
\begin{proof}
	The group \(G\) has a classifying space \(BG\) with \(H_k (G;\mathbb{Z}) \cong H_k (BG;\mathbb{Z})\).  
	A homomorphism \(\phi\colon G\to K\) induces a continuous map \(B\phi\colon BG\to BK\).
	Thus  \(i\colon  H\hookrightarrow G\) induces \(Bi\colon BH \to BG\) and \(r\colon  G\to H\) induces \(Br\colon BG\to BH\).
	The equality \(r\circ i = \mathrm{id}_H\) implies
	$
	Br\circ Bi = B(r\circ i) = B(\mathrm{id}_H) = \mathrm{id}_{BH}.
	$
	Now take homology with integer coefficients:
	\[
	(Br)_* \circ (Bi)_* = (\mathrm{id}_{BH})_* = \mathrm{id}_{H_k (BH;\mathbb{Z})} = \mathrm{id}_{H_k (H;\mathbb{Z})}.
	\]
	
	Thus we have an injection
	$
	(Bi)_*\colon  H_k (H;\mathbb{Z}) \hookrightarrow H_k (G;\mathbb{Z}) \tn{ for all } k.
	$
\end{proof}

\begin{theorem}[\citenum{spSubgroupsGeometric}, Lemma 1.4]\label{separable and virtually embedded}
	If $G=\pi_1(M)$ of a manifold with universal cover $\tm$, then $H<G$ is separable if and only if each compact subset of $\tm/H$ embeds in an intermediate finite cover of $M$. 
\end{theorem}

We rephrase [\citenum{bhwHyperplaneSection}, Lemma 1.7 and the last paragraph],[\citenum{bwBoundaryCriterionCubulation}, THEOREM 6.2], and \cite{cdwVirtuallySpecial3-manifolds} as follows. 
\begin{lem}[\citenum{bhwHyperplaneSection,bwBoundaryCriterionCubulation}]\label{totally geodesic hypersurface separable}
	Let $M =G  \backslash \hs$ be an arithmetic manifold of type I. The $\pi_1(M)$ is virtually special and the fundamental group of a totally geodesic hypersurface $F\hookrightarrow M$ is a virtual retract. 	
\end{lem}

\begin{theorem}[\citenum{cdwVirtuallySpecial3-manifolds}, Theorem 1.3]\label{virtually special virtual retract}
Let $X$ be a compact, virtually special cube complex and suppose that $\pi_1(X)$ is hyperbolic relative to a collection of finitely generated abelian subgroups. Then every relatively quasiconvex subgroup of $\pi_1(X)$  is a virtual retract.
\end{theorem}

\begin{theorem}[\citenum{bhwHyperplaneSection}, Corollary 1.3]\label{totally geodesic hypersurface virtual homology}
	Let $M=G \backslash\hs$ be an arithmetic manifold and $P\subset \hs$ be a $G $-hyperplane. Then, there exists a finite cover $\hat{M}$ of $M$ such that $P$ projects to an embedded submanifold $F\hookrightarrow M$ and 
	$$H_k(F, \Z) \rightarrow H_k(\hat{M},\Z)$$
	is injective for every integer $k\geq 0$. 
\end{theorem}

\begin{cor}[\citenum{bhwHyperplaneSection}, Corollary 1.6]
	Let $M =G  \backslash \hs$ be an arithmetic hyperbolic manifold and $F_1$ and $F_2$ be two totally geodesic immersed submanifolds in $M$. Assume $F_1$ and $F_2$ transversally intersect in at least one point. Then there exists a finite cover $\mhat$ of $M$ and two connected components $\hat{F}_1$ and $\hat{F}_2$ of the preimages of $F_1$ and $F_2$ in $\mhat$ such that $\hat{F}_1$ and $\hat{F}_2$ are both embedded in $\mhat$ and their intersection $\hat{F}_1\cap \hat{F}_2$ is connected and non-trivial in $H_*(\mhat)$. 
\end{cor}
\subsection{Strongly filling subgroups and cubulations}
\begin{theorem}\label{strongly filling applications}
	If $S$ is a \poin, strongly filling hypersurface of a closed negatively curved manifold $M$, such that $\pi_1(S)$ is a quasi-convex, codimension-$1$ subgroup of $\pi_1(M)$, then $\pi_1(M)$ acts properly and cocompactly on a \catzero cube complex with one family of hyperplanes, and $\pi_1(M)$ is virtually special and hence embeds in $\spl_d(\Z)$ for some $d\in \N$. 
\end{theorem}
\begin{proof}
	Since $M$ is a closed closed negatively manifold, $\pi_1(M)$ is word-hyperbolic. If $S$ is strongly filling and $\pi_1(S)$ is a quasi-convex, codimension-$1$ subgroup of $\pi_1(M)$, then $\{\pi_1(S)\}$ satisfies the conditions in \Cref{cube complex cocompact} and \Cref{bergeron-wise}.
    Thus $\pi_1(M)$ acts properly and cocompactly on a \catzero cube complex $X$.
    By \Cref{virtual haken conjecture}, $G$ is virtually special, i.e. a finite-index subgroup is the fundamental group of special cube complex $X$ in the sense of \cite{hwSpecialCubeComplex}.
    Therefore $\pi_1(M)$ is a right-angled Artin group (RAAG).
    Since RAAGs are known to be linear over $\Z$, $\pi_1(M)$ is virtually linear over $\Z$. Using induced representations, $\pi_1(M)$ is linear over $\Z$. 
	
	The orbits of the hyperplanes correspond to the image $\{\pi_1(S)\}$ under the natural action by $\pi_1(M)$. Thus there is only one family of hyperplanes. 
\end{proof}

A CAT(0) cube complex $X$ is \textit{essential}, if for each hyperplane $w$, each of the associated halfspaces contains points in $X$ arbitrarily far from $w$. If a group $G$ acts by automorphism on $X$, the action is \textit{essential} if for any point $x$ in the zero-skeleton of $X$, and each hyperplane $w$, each of the associated halfspaces contains points in $G\cdot x$ arbitrarily far from $w$. 

The cube complex $X$ is \textit{hyperplane-essential} if each hyperplane $w$, regarded itself as a CAT(0) cube complex, is essential. The action of $G$ on $X$ is \textit{hyperplane-essential} if each hyperplane $w$ has  the property that the stabilizer of $w$ acts essentially on $w$.

The stabilizer of a hyperplane correspond to conjugation of $H \cong \pi_1(S)$. Since $S$ is compact and negatively curved by \Cref{nearly geodesic hypersurfaces}, $\pi_1(S)$ is again a word-hyperbolic group. As in \cite{hxhNearlyGeodesicFilling}, we can again deduce that $X$ is essential and hyperplane-essential. 
\begin{cor}\label{strongly filling implies virtual homology injects}
If $S$ is a \poi strongly filling hypersurface of a closed hyperbolic manifold $M$ such that $\pi_1(S)$ is a quasi-convex, codimension-$1$ subgroup of $\pi_1(M)$, then every quasi-convex subgroup is a virtual retract of $\pi_1(S)$. Thus if $S$ is a \poi hypersurface with quasi-convex subgroup $\pi_1(S)$, there are finite covers $\hat{M}$ of $M$ and $\hat{S}$ of $S$ such that 
		$$H_k(\hat{S}, \Z) \rightarrow H_k(\hat{M},\Z)$$
		is injective for each $k \geq 0$. 
\end{cor}
\begin{proof}
    By \Cref{strongly filling applications}, $\pi_1(M)$ is virtually special. By \Cref{virtually special virtual retract}, every quasi-convex subgroup is a virtual retract of $\pi_1(M)$. The conclusion now follows from \Cref{retraction implies injection}. 
\end{proof}
Thus hypersurfaces $S$ in a hyperbolic manifold $M$ with $\sff_{[S]}<1$ have quasi-convex fundamental groups and hence are virtual retract. This implies that $S$ virtually homologically injects and completes the proof of \Cref{asymptotically geodesic implies virtually special}. 
\section{Nearly geodesic and \ag hypersurfaces}\label{ggt}
\subsection{Nearly geodesic hypersurfaces}
Inspired by the works on almost Fuchsian and nearly Fuchsian surfaces of hyperbolic $3$-manifolds, we make the following definition. 
\begin{defi}[Nearly geodesic hypersurface]
	Let $N$ be a Riemannian manifold. 
	A closed, smoothly immersed hypersurface $S \hookrightarrow N$ is \emph{\nege } if $\Vert \sff_{S}\Vert_{L^\infty(S)} <1$. 
\end{defi}
Thurston observes that \nege  hypersurfaces in a hyperbolic manifolds are \poi (see [\citenum{lcSmallCurvatureSurface}, Theorem 5.1]). El Emam and Seppi establish some further geometric and topological properties which we recall below. 

\begin{lem}[\citenum{esGaussMapEquivariantImmersion}, Section 4]\label{lemma_embeddedness}
Let $S\subset M$ be a closed immersed \almostgeo hypersurface, and let $\widetilde S\subset \mathbb H^{n+1}$ be a connected lift.
Then we have the following. 
\begin{enumerate}
\item $\widetilde{S}$ is uniformly negatively curved. 
\item $\widetilde S$ is properly embedded in $\mathbb H^{n+1}$ and separates $\mathbb H^{n+1}$ into two connected components.
\item The asymptotic boundary $\partial_\infty \widetilde S\subset S^n_\infty$ is a topological $(n-1)$-sphere, and $\partial_\infty \mathbb H^{n+1}\setminus \partial_\infty \widetilde S$ consists of two connected components.
\end{enumerate}
\end{lem}

\begin{proof}
Let $S\hookrightarrow M$ be a closed hypersurface such that $\|\sff_{S}\|_{L^\infty(S)}<1$. Consider a lift $\iota\colon  \ts \hookrightarrow \hs$ which is a Riemannian covering of $S$.
[\citenum{esGaussMapEquivariantImmersion}, Remark 4.3] points out that $\ws$ is negatively curved. We provide some more detail for convenience. 
	By the compactness of $S$, there exists a $\delta\in (0,1]$ such that $\|\sff_{S}\|_{L^\infty(S)}\leq 1-\delta$. By the Gauss equation in hyperbolic manifolds, for any $2$-plane $\sigma\in T_xS$, the sectional curvature $\sec(\sigma) = -1 + \det(\sff_S|_\sigma)$. Since all the principal curvatures of $S$ satisfy $k_i \leq 1-\delta$, $\det(\sff_S|_\sigma) = k_ik_j\leq (1-\delta)^2$ and therefore $\sec(\sigma) \leq -2\delta+\delta^2$. 

Since $S$ is closed, the induced metric on $\widetilde S$ is complete. 
By \cite[Proposition 4.15]{esGaussMapEquivariantImmersion}, any complete immersion in $\mathbb H^{n+1}$ with principal curvatures in $(-1,1)$ is a proper embedding. Hence $\widetilde S$ is properly embedded.

Moreover, by \cite[Proposition 4.18]{esGaussMapEquivariantImmersion}, the normal exponential map along $\widetilde S$ defines a foliation of $\mathbb H^{n+1}$ by equidistant hypersurfaces. In particular, $\widetilde S$ separates $\mathbb H^{n+1}$ into two connected components.

Furthermore, by \cite[Remark 4.19]{esGaussMapEquivariantImmersion}, the inclusion $\widetilde S\hookrightarrow \mathbb H^{n+1}$ is a quasi-isometric embedding, and extends continuously to the boundary at infinity. The image $\partial_\infty \widetilde S$ is a topological $(n-1)$-sphere. Finally, by \cite[Proposition 4.20]{esGaussMapEquivariantImmersion}, the complement $\partial_\infty \mathbb H^{n+1}\setminus \partial_\infty \widetilde S$
has two connected components.
\end{proof}

\begin{lem}\label{nearly geodesic hypersurfaces}
	Let $M$ be a closed hyperbolic $(n+1)$-manifold and  $S\hookrightarrow M$ a closed, \almostgeo hypersurface. Then $\pi_1(S)\rightarrow \pi_1(M)$ is injective, quasi-convex, and codimension-$1$. 
\end{lem}
\begin{proof}
Let $\ts\subset \mathbb{H}^{n+1}$ be a lift of $S$. From Lemma~\ref{lemma_embeddedness} it follows that $\ts $ is properly embedded in $\mathbb{H}^{n+1}$, $\partial \ts \cong S^{n-1}$, and $\ts$ is a Cartan-Hadamard space. Since $\ts$ is contractible, $\pi_1(S)\rightarrow \pi_1(M)$ must be injective. 
	
    By \cite[Remark 4.19]{esGaussMapEquivariantImmersion}, the inclusion $\widetilde S\hookrightarrow \mathbb H^{n+1}$ is a quasi-isometric embedding. This implies that for any two points $p, q \in \ts$, the hyperbolic geodesic connecting $p$ and $q$ is within uniform distance (depending on $K(\delta)$) to $\ts$. Since $S$ is compact, $\pi_1(S)$ is uniformly quasi-isometric to the universal cover $\ts$, and thus $\pi_1(S) \rightarrow\pi_1(M)$ is quasi-convex. 

    To establish that $\pi_1(S)$ is a codimension-$1$ subgroup of $\pi_1(M)$, we first observe from the previous lemma that $\hs-\ts$ consists of two components.

    Since both $S$ and $M$ are compact, their universal covers are quasi-isometric to their Cayley graphs $\cg(\pi_1(S)), \cg(\pi_1(M))$. Since $S$ is orientable, $\pi_1(S)$ stablizes the two components of $  \hs  - \ts$. Thus $ \cg(\pi_1(M))-\cg(\pi_1(S))$ under the action of $\pi_1(S)$ has two components, verifying that $\pi_1(S)$ is a codimension-$1$ subgroup of $\pi_1(M)$. 
\end{proof}
Since the fundmental groups of arithmetic manifolds of type I are virtually special (\Cref{totally geodesic hypersurface separable}), by \Cref{strongly filling implies virtual homology injects}, we have the following. 
\begin{cor}\label{almost geo of arithmetic are virtually embedded}
	A closed \nege  hypersurface of an arithmetic manifold of type $I$ is virtually embedded and virtually homologially injects. 
\end{cor}
The following is a generalization of Uhlenbeck \cite{ukclosedminSurfaceHyperbolic}.
\begin{theorem}\label{totally geodesic area-minimizing}
	Let $S$ be a closed hyperbolic $n$-manifold  and $M$ a hyperbolic $(n+1)$-manifold. If a \poi immersion $\iota\colon  S \to M$ is totally geodesic, then it is area-minimizing in the homotopy class. 
\end{theorem}
\begin{proof}
Let $\iota_0:S\to M$ be a totally geodesic $\pi_1$-injective immersion, and let $H=\iota_{0*}\pi_1(S)\le \pi_1(M)\le \Isom^+(\mathbb H^{n+1})$. Since $\iota_0$ is totally geodesic, $H$ preserves a hyperplane $P_0\subset \mathbb H^{n+1}$. Set
\[
M' = H\backslash \mathbb H^{n+1}, \quad S_0 = H\backslash P_0 \subset M'.
\]
Then $ M'$ is a convex cocompact hyperbolic \((n+1)\)-manifold, $S_0$ is a closed totally geodesic embedded hypersurface, and $\iota_0$ lifts to an isometric identification $S\cong S_0$. Let $\iota:S\to M$ be any immersion homotopic to $\iota_0$. Then $\iota$ lifts to a map $\tilde \iota:S\to M'$ homotopic to the inclusion $S_0\hookrightarrow M'$.

Let $\pi:M'\to S_0$ be the nearest-point projection which is $1$-Lipschitz. The $n$-Jacobian of $d\pi$ is strictly less than $1$ away from $S_0$.
Therefore
\[
\area((\pi\circ \tilde \iota)(S))\leq \area(\tilde \iota(S)),
\]
with equality only if $\tilde \iota(S)\subset S_0$. Identifying $S_0$ with $S$ via $\iota_0$, the map $\pi\circ \tilde \iota$ becomes a
self-map $f:S\to S$ homotopic to the identity, hence $\deg f=1$. Since $S$ is
hyperbolic, any degree-$1$ map $S\to S$ has area at least $\area(S)$.
Thus
\[
\area(\iota(S))=\area(\tilde \iota(S))\geq
\area((\pi\circ \tilde \iota)(S))\geq \area(S)=\area(\iota_0(S)).
\]
If equality holds, then $\tilde \iota(S)\subset S_0$, so $\iota$ is totally geodesic. Hence $\iota_0$ is area-minimizing in its homotopy class.
\end{proof}

\subsection{Asymptotically geodesic hypersurfaces are asymptotically dense in $\gnmo M$}
\begin{defi}[Asymptotically geodesic hypersurfaces]
	Let $M$ be a Riemannian manifold. We call a sequence of smoothly immersed hypersurfaces $[S_i] \hookrightarrow M$ \emph{\ag} if $\Vert \sff_{S_i}\Vert_{L^\infty(S_i)}\to 0$ as $i\to\infty$, where $\sff_{S_i}$ denotes the second fundamental form of $S_i$.
\end{defi}
Our goal in this subsection is to prove \ag hypersurfaces are asymptotically dense in the Grassmann bundles and eventually strongly filling. The following lemma controls the deviation between a pair of tangent geodesic and $k$-geodesic (a unit-speed \(C^2\)  curve whose curvature is $\leq k$). 
\begin{lem}\label{controlling distance between k-geodesic and geodesic}
	In $\hs$, let \(a(t)\) be a unit–speed geodesic, \(b(t)\) a unit–speed \(C^2\) curve whose geodesic curvature satisfies
	\(\Vert\kappa(t)\Vert \le k\) for all \(t\),
	and \(a(0)=b(0)\), \(a'(0)=b'(0)\).
	Then  for all \(t\ge 0\),
	\[
	d\bigl(a(t),b(t)\bigr)\;\le\;k\,(\cosh t - 1). 
	\]
\end{lem}
\begin{proof}
	We first setup in Fermi coordinates along \(a\). Take a parallel orthonormal normal frame \(\{e_1(t),\dots,e_{n}(t)\}\) along \(a(t)\).
	For small normal vectors \(u\in\mathbb R^{n}\), points near \(a(t)\) can be written as
	\[
	x = \exp_{a(t)}\Bigl(\sum_i u_i e_i(t)\Bigr).
	\]
	
	Assume \(b(t)\) stays in a normal neighborhood of \(a\) and write
	\[
	b(t) = \exp_{a(t)}\bigl( u(t) \bigr),\qquad
	u(t)\in\mathbb R^{n},
	\]
	so \(u(t)\) is the ``normal displacemen'' of \(b\) from \(a\) at time \(t\). The acceleration of \(b\) decomposes into tangential and normal parts:
	\[
	\nabla_{\dot b}\dot b = \kappa(t)\,\nu(t),
	\]
	where \(\kappa(t)\) is the geodesic curvature and \(\nu(t)\) is a unit normal.
	
	In the Fermi coordinates above, the normal component of the relative displacement \(u(t)\) satisfies a second–order ODE of the form
	\[
	u''(t) - u(t) = F(t),
	\]
	where the homogeneous operator \(u'' - u\) is the Jacobi operator for normal variations in curvature \(-1\), and 
	the inhomogeneity \(F(t)\) encodes the normal acceleration difference between \(a\) (which has zero normal acceleration) and \(b\); its size is controlled by the curvature bound:
	\[
	\|F(t)\|\;\le\;\Vert\kappa(t)\Vert\;\le\;k.
	\]
	
	Since \(a(0)=b(0)\) and \(a'(0)=b'(0)\), we have
	$
	u(0)=0, u'(0)=0.
	$
	Thus each component of \(u\) satisfies the scalar ODE
	\[
	u'' - u = f(t),\quad |f(t)|\le k,\quad u(0)=u'(0)=0.
	\]
	
	The solution with those initial data is
$
	u(t)=\int_0^t \sinh(t-s)\, f(s)\,ds.
	$
	
	Taking norms and using \(|f(s)|\le k\):
	\[
	\|u(t)\|
	\;\le\;
	\int_0^t \sinh(t-s)\,|f(s)|\,ds
	\;\le\;
	k\int_0^t \sinh(t-s)\,ds
	=
	k(\cosh t - 1).
	\]
	
	The above inequality is valid globally if we parameterize \(b\) via the normal exponential map along \(a\). Thus we get the uniform estimate
$
	d\bigl(a(t),b(t)\bigr) \;\le\; k\,(\cosh t - 1)
	\quad\text{for all }t\ge 0.
$
\end{proof}
All curves are assumed to have unit-speed parameterizations in the paper. 
By standard facts for the second fundamental forms and the decomposition of the curvature of a curve in a submanifold into tangential and normal components, if a submanifold has principal curvature $\leq k$, then every intrinsic geodesic has curvature $\leq k$.

The proof of the next proposition is inspired by the proof of [\citenum{cmnCountingminimal}, Proposition 4.2]. Their proof relies on the compactness of minimal surface theory, and firstly shows that the lifts $\tsj \subset \hy$ of asymptotically geodesic surfaces $S_j$ graphically converge to a totally geodesic plane $P$, and then shows that $\partial \tsj$ converges to $\partial P$. Instead, our proof relies on elementary properties of hyperbolic geometry, Cartan-Hadamard space, curvature comparison, and Hausdorff convergence (see [\citenum{bpLecturesHyperbolicGeo}, Proposition E.1.2.]). 	In a locally compact metrizable space $X$, a sequence $C_j $ of closed sets \emph{converges} to a closed set $C$ (denoted by $\lim\limits_{n\rightarrow \infty}C_j =C$) if the following two conditions are fulfilled: 
\begin{enumerate}
	\item if $x\in X$ is such that there exists a subsequence $C_{n_j}$ of ${C_j }$ and $x_j \in C_{n_j}$ with $x_j \rightarrow x $ in $X$, then $x\in C$; 
	\item given $x \in C$, there exists $x_j \in C_j$ for all $j$ such that $x_j\rightarrow x$ in $X$. 
\end{enumerate}
Since there is no preferred metric on $\partial \hs$, to discuss the topological convergence we can just take the hyperbolic ball model as a subset of Euclidean space and metric on $\partial \hs$ induced by the Euclidean metric. 
\begin{defi}
	A $k$-cone around $(\tx, \tpai)\in \gnmo \hs$ is the union of all unit-speed $k$-geodesics in $\hs$ tangent to $(\tx, v)$ where $v\in \tpai$. 	
\end{defi}

Let $N$ be a cusped hyperbolic $(n+1)$-manifold  and $\eta$ a small constant so that the complement of $\neta\coloneqq \{x\in N: \inj(x)\geq \eta\}$ in $N$ is a union of embedded cusp neighborhoods (each neighborhood is homeomorphic to $\T^{n}\times [0, \infty)$).
The thin part $N-\neta$ consists of cusps and cannot contain a quasi-geodesic. Thus for an \almostgeo hypersurface $S\hookrightarrow N$, $S \cap \neta \neq \emptyset $. 
Recall that by a sequence of hypersurfaces we mean a sequence of hypersurfaces that are distinct in the equivalence relation generated by homotopy and commensurability. 
\begin{prop}\label{ag implies asymptotically dense}
	Let $N$ be a finite-volume hyperbolic $(n+1)$-manifold. 
	A sequence of closed \ag hypersurfaces $S_i \hookrightarrow N$ is asymptotically dense in the Grassmann bundles $\gnmo N$.
\end{prop}
\begin{proof}
	Recall that $\neta$ is a compact core of $N$ so that $N-\neta$ consists of cusp neighborhoods. A subsequence of $(p_i, \Pi_i)\in \gnmo S_i \cap \gnmo \neta$ converges to $(p, \Pi)\in \gnmo \neta$, which we denote by sub-index $j$. 
	
	Choosing a fixed fundamental domain $D \subset \hs$ for $N$, we lift $(p_j, \Pi_j)$ to $(\widetilde{p_j}, \widetilde{\Pi_j}) \in \gnmo D$ which converges to $(\widetilde{p}, \widetilde{\Pi})$. Let $\nu_j$ be the unit normal along $S_j$. Since $|\nabla_{S_j}\nu_j|=|\sff_{S_j}|\eqqcolon k_j$, the assumption $|\sff_{S_j}|_{L^\infty(S_j)}\to 0$ implies that the variation of normals is uniformly controlled on $S_j$. Thus the tangent planes to $(\widetilde{p_j}, \widetilde{\Pi_j})$  are uniformly close to the fixed hyperplane $P\subset \hs$ tangent to $(\widetilde{p}, \widetilde{\Pi})$. 
	Let $P_j$ be the (totally geodesic) hyperplane in $\hs$ tangent to $\widetilde{\Pi_j}$. For $j$ large, the nearest point projection $\pr_P^{\tsj}$ from $\tsj$ onto $P$ is a composition of first projection $\pr_{P_j}^{\tsj}$ onto $P_j$ and then projection $\pr_P^{P_j}$ from $P_j$ to $P$. The norm $k_j$ on the second fundamental forms of $\tsj$ controls the projection $\pr_{P_j}^{\tsj}$, and the distance $d_{\gnmo \hs}(\widetilde{\Pi_j}, \widetilde{\Pi})$ controls the projection $\pr_P^{P_j}$. The projection $\pr_P^{\tsj}$ is a diffeomorphism on any compact subset $K\cap \tsj$. Thus $\tsj$ is a graph of $P$ with small $C^1$ norm. 
	
We first prove that $\lim\limits_{j\rightarrow \infty}\partial \tsj =\partial P$. Set $\tilde{p}$ as the origin of the hyperbolic ball model.
	Take $q_j \in \partial \tsj$, where $\partial \tsj$ is a topological $(n-1)$-sphere. Since $\tsj\subset \hs$ with respect to the intrinsic metric is a Cartan-Hadamard space, there exists a one-to-one correspondence between the unit vectors in $T^1_{\tpj}{\tsj}$ and the points in $\partial \tsj$. There is also a unique intrinsic geodesic ray $[\widetilde{p_j},q_j]_{\tsj}(t)$ in $ \tsj$ such that $[\widetilde{p_j},q_j]_{\tsj}(0)=\widetilde{p_j}$ to $[\widetilde{p_j},q_j]_{\tsj}(\infty)=q_j$. By \Cref{controlling distance between k-geodesic and geodesic}, the collection of all $k_j$-geodesic rays tangent to vectors in $T^1_{\tpj}{\tsj}$ forms a $k$-cone centered at $(\tpj, \tpaij)$ and must contain $\tsj$. In particular, let $b_j$ be the geodesic of $\hs$ tangent to $[\widetilde{p_j},q_j]_{\tsj}$ at $\widetilde{p_j}$. \Cref{controlling distance between k-geodesic and geodesic} controls the hyperbolic distance between $[\widetilde{p_j},q_j]_{\tsj}(t)$ and $b_j(t)$. 
	
	If we parallel transport $b_j'(t)$ from $\tpj$ to a vector in $T_{\tilde{p}}P$, we obtain a vector $w_j$. Consider the geodesic $\beta_j(t)$ in $P$ tangent to $w_j$ such that $\beta_j(0)=\tilde{p}$. 
	By the triangle inequality, $d(\beta_j(t), [\widetilde{p_j},q_j]_{\tsj}(t))$ is controlled by $d(\tilde{p}, \tpj), k_j,$ and $t$. Taking the standard spherical distance $d_E$ on $\partial \hs \subset \mathbb{E}^{n}$, we see that $d_E([\widetilde{p_j},q_j]_{\tsj}(\infty), \partial P) \rightarrow 0$ as $j$ tends to infinity. This shows that every point of $\partial \tsj$ limits to a point on $\partial P$. Conversely, let $q \in \partial P$. There exists a unique geodesic ray $\beta(t)$ from $\tilde{p}$ to $q$, lying on $P$. If we parallel transport $\beta'(0)$ to a tangent vector at $T_{\tpj}\tsj$, the same argument above shows that $q$ is a limit point of a sequence of elements in $\partial \tsj$. 
	
	The remaining of the proof that asymptotically geodesic hypersurfaces are asymptotically dense in $\gnmo N$ is similar to \cite{alAsymptoticallyGeodesicsurfaces,hxhNearlyGeodesicFilling}, and thus we only briefly recall. Since $S_j$ is smoothly immersed, $\gnmo S_j$ is a closed, connected subset of $\gnmo N$, and thus converges to a closed, connected subset $ C$ of $\gnmo N$ in the Hausdorff metric. Since $\gnmo S_j$ are asymptotically geodesic, $C$ is foliated by totally geodesic hyperplanes. Since $C \subset \gnmo N$ is connected, $C$ is equal to one hyperplane or the entire $\gnmo N$ by Ratner-Shah theorem \cite{rmRatner'stheorem,snRatnerTheorem}. But if $C$ is one hyperplane $P$ (which must correspond to a properly immersed totally geodesic hypersurface), we deduce as in \cite{alAsymptoticallyGeodesicsurfaces,hxhNearlyGeodesicFilling} that for any $\epsilon>0$, there exists an $n_0\in \N$ such that for all $j>n_0$, 
	all $S_j$ are supported in an $\epsilon$-neighborhood of $P$, and thus must be eventually commensurable to $P$. 
\end{proof}
\Cref{ag implies asymptotically dense} generalizes \cite{alAsymptoticallyGeodesicsurfaces} and \cite{hxhNearlyGeodesicFilling}.

\begin{prop}\label{ag implies eventually strongly filling}
	A sequence of pairwise distinct classes of closed asymptotically geodesic hypersurfaces $S_i\hookrightarrow N$ is eventually strongly filling. 
\end{prop}
\begin{proof}
	Suppose not. Then we extract a subsequence $S_j$ which is asymptotically dense in $\gm N$ but not strongly filling. For each $S_j$, all its lifts in $\hs$ do not separate $p_j, q_j\in \partial_{\infty} \hs$. Let $\gamma_j \coloneqq \pi([p_j,q_j])$. Take $v_j \in T^1\gamma_j \cap \neta$ which converges up to a subsequence to $v \in T^1\neta$ by the compactness of $\neta$. In the universal cover, there is a lift $\tilde{v}$ tangent to a geodesic $\tgamma$. The sequence $S_j$ is asymptotically dense and asymptotically geodesic, and thus admits a subsequence of lifts $\widetilde{S_j}$ which asymptotically converge to a hyperplane $P$ orthogonal to $\tgamma$ somewhere. 
	
	By  the previous proof, $\partial\widetilde{S_j}$ converges to $\partial P$. 
	 Thus the limit set of $\widetilde{S_j}$ must eventually be linked with the endpoints of $\tgamma$. Hence for sufficiently large $j$, $S_j$ separates the endpoints of $\gamma_j$, which is a contradiction. 
\end{proof}
\begin{rem}
	In a Euclidean $3$-torus $\T^3$, by taking advantage of the scaling in the universal cover, we can construct a sequence of \ag closed minimal surfaces in $\T^3$. Thus its principal curvature spectrum has no gap at $0$. Its fundamental group   $\pi_1(\T^3)\cong \Z^3$ is also linear over $\Z$. 
\end{rem} 
Recall that in \Cref{principal curvature spectrum}, we define a functional on the set of classes of closed, smoothly immersed hypersurfaces in $M$ by $
	\Lambda_M(	[S])=   \inf\limits_{S'\in[S]} \|\sff_{S'}\|_{L^{\infty}}\eqqcolon \sff_{[S]}.$
\begin{lem}\label{zeros of spectrum}
	If $\sff_{[S]}=0$, then $[S]$ contains a totally geodesic representative. 
\end{lem}
\begin{proof}
	Let $S_i\hookrightarrow M$ be a sequence of surfaces homotopic and  commensurable to $S$, such that $k_i\coloneqq \|\sff_{S_i}\|_{L^{\infty}(S_i)}\to 0$. Let $(x_i,\tpai_i)\in \gnmo S_i$. Then up to taking a subsequence, $(x_j, \tpai_j)$ converges to $(x,\tpai)\in \gnmo M$. Taking lifts in $\hs$, we have $(\txj, \tpaij)\to (\tx, \tpai)$. By using $k_j$-cones centered at $(\txj, \tpaij)$ and ideas in the proof of \Cref{ag implies asymptotically dense}, we again have that $\partial \tsj \to P$ for some totally geodesic plane $P$. The assumption that $S_j$ homotopic and commensurable to $S$ implies that $\partial\tsj$ are all equal. Thus the only possibility is that $\partial\tsj=\partial P$ and $S_j$ is homotopic to a closed totally geodesic hypersurface. 
\end{proof}
We now prove \Cref{asymptotically geodesic implies virtually special} and  \Cref{spectrum and strongly filling}. 
\begin{cor}
	Let $M$ be a closed hyperbolic manifold. There exists a constant $C< 1$ depending on $M$ such that if $\image(\Lambda_M) \cap (0,C]$ is non-empty or if $0$ is of infinite multiplicity in $\image(\Lambda_M)$, then $\pi_1(M)$ is virtually special, and every closed hypersurface $S$ such that $\sff_{[S]}<1$ is virtually embedded and virtually homologically injects. 
\end{cor}
\begin{proof}
		Let $S_i$ be  a sequence of asymptotically geodesic hypersurfaces in $M$ (equivalently, $\sff_{[S_i]}\rightarrow 0$). All but finitely many $S_i$ are strongly filling by \Cref{ag implies eventually strongly filling}. There is a subsequence $S_j$ such that for all $j$, either $\sff_{[S_j]}>0$ but $\sff_{[S_j]}\rightarrow 0$, or $\sff_{[S_j]}=0$. 		
		
		Thus there exists a constant $C$ such that  if $\sff_{[S]}<C$, then $S$ is strongly filling. By \Cref{nearly geodesic hypersurfaces}, if $S \hookrightarrow M$ is a closed hypersurface such that $\sff_{[S]}<1$, then $\pi_1(M)$ is quasi-convex and codimension-$1$.  The conclusion now follows from \Cref{strongly filling applications}. 
\end{proof}

\section{Constructions}\label{constructions}
This section discusses the proof of \Cref{arithmetic of type I contains nice hypersurfaces}, which is derived from the following theorem. 

\begin{theorem}\label{thm_main_of_sec4}
	Let $M$ be a closed arithmetic hyperbolic $(n+1)$-manifold of type I. Then $M$ contains a sequence of \poi hypersurfaces which are asymptotically geodesic, equidistributing in the Grassmann bundle $\gm M$, and not homotopic to totally geodesic hypersurfaces. 
\end{theorem}

First, in \Cref{pairs intersecting}, we prove that all but finitely many pairs of \ag hypersurfaces intersect, providing the foundation for subsequent construction. The above theorem is then proved in Sections~\ref{subsection_construction of pi1-injective hypersurface}-\ref{subsection_asymptotically geodesic}.
In \Cref{subsection_construction of pi1-injective hypersurface}, we present the construction and prove the $\pi_1$-injectivity. The equidistribution and asymptotically geodesic property are established in \Cref{subsection_equidistribution} and \Cref{subsection_asymptotically geodesic}, respectively. 

Here we restate and prove \Cref{1.2}. 
\begin{prop}\label{pairs intersecting}
	Let $N$ be a finite-volume hyperbolic manifold and $S_i$ a sequence of closed \poi \ag hypersurfaces. Then all but finitely many pairs of $S_i, S_j$ intersect into a codimension-$2$ submanifold that is strongly filling in $S_i$ or $S_j$. 
\end{prop}
\begin{proof}
	By \Cref{ag implies eventually strongly filling}, all but finitely many closed \ag hypersurfaces $S_i$ are strongly filling. Let $S_i$ be a strongly filling hypersurface. Since any $S\coloneqq S_j$ is \poin, for every $\gamma \in \pi_1(S)$ and one of its lifts $\tgamma \subset \hs$, there exists a lift $\tsi \subset \hs$ such that $\partial \tsi $ separates $\partial \tgamma$. Denote a lift of $S$ which contains $\tgamma$ as $\ts$. Since $\tgamma(\infty)$ and $\tgamma(-\infty)$ belong to different components of $\partial \hs -\partial \tsi$, they must also belong to different components of $\partial \ts-\partial \tsi$. In other words, $\ts \cap \tsi$ is a codimension-$1$ submanifold of $\ts$ which separates $\tgamma(\infty)$ and $\tgamma(-\infty)$. This proves that $S_i\cap S$ is a strongly filling submanifold of $S$.
	
\end{proof}

\subsection{The $\pi_1$-injectivity}\label{subsection_construction of pi1-injective hypersurface}

In this subsection, we establish the construction of hypersurfaces in \Cref{thm_main_of_sec4} and prove the $\pi_1$-injectivity. 

\subsubsection{Lifting a pair of intersecting totally geodesic hypersurfaces to suitable finite covers}
\begin{lem}
Let $M$ be a closed arithmetic hyperbolic manifold of type I. 
    After passing to a finite cover $\hat{M}$ of $M$ if necessary, there are two intersecting embedded totally geodesic lifts $\hat{\Sigma}_1$ and $\hat{\Sigma}_2$ in $\hat{M}$ such that $L=\hat{\Sigma}_1\cap\hat{\Sigma}_2$ is transverse and nonseparating in both $\hat{\Sigma}_1$ and $\hat{\Sigma}_2$.
\end{lem}
\begin{proof}
	Since $M$ is arithmetic of type I, $M$ contains a sequence of \cbility classes of closed totally geodesic hypersurfaces $\{\Sigma_i\}$. 
	By \Cref{pairs intersecting}, all but finitely many pairs of them intersect. For simplicity of notation, assume that $\Sigma_1$ nontrivially intersects $\Sigma_2$ (we can take $\Sigma_1=\Sigma_2)$. By \Cref{totally geodesic hypersurface virtual homology}, there is a finite cover $\hat{M}^0$ of $M$ such that $\Sigma_1$ lifts to an embedded finite cover $\hat{\Sigma}_{1}^0 \hookrightarrow \hat{M}^0$ that intersects a cover $\hat{\Sigma}_{2}^0 $ of $\Sigma_2$ in a codimension-$2$ subsurface $\hat{\Sigma}_{1}^0 \cap \hat{\Sigma}_{2}^0 \eqqcolon L^0$. By taking another finite cover of $\hat{M}^0$ if necessary, we lift both $\hat{\Sigma}_{1}^0, \hat{\Sigma}_{2}^0$ to a pair of transversally intersecting embedded totally geodesic hypersurfaces. For simplicity of notations, we assume $\hat{M}^0$ is a finite cover of $M$ with a pair of  embedded totally geodesic hypersurfaces $\hat{\Sigma}_{1}^0, \hat{\Sigma}_{2}^0$ which transversally intersect into a codimension-$2$ surface $L^0$. 
	
	Since totally geodesic submanifolds of an arithmetic hyperbolic manifold \oto are arithmetic \oton, $\hat{\Sigma}_{1}^0, \hat{\Sigma}_{2}^0$ are arithmetic, and $L^0$ is an arithmetic hypersurface of $\hat{\Sigma}_{1}^0, \hat{\Sigma}_{2}^0$. By \Cref{totally geodesic hypersurface virtual homology}, $\hat{\Sigma}_{1}^0, \hat{\Sigma}_{2}^0$ have finite covers $\hat{\Sigma}_1, \hat{\Sigma}_2 $ in $ \hat{M}$ that covers $\hat{M}^0$, such that $L^0$ lifts to $L$ nonseparating in both $\hat{\Sigma}_1, \hat{\Sigma}_2$ since $L$ represents homology. 
\end{proof}

\subsubsection{Cutting and pasting to obtain the $\pi_1$-injectivity}
By the discussion below, we may further arrange that the dihedral angle between $\hat{\Sigma}_1$ and $\hat{\Sigma}_2$ along $L$ is bounded below by a fixed constant $\theta>0$ (Lemma~\ref{lem_angle density}).

We then cut both hypersurfaces along $L$ and reglue them by exchanging the complementary pieces. In the universal cover $\mathbb{H}^{n+1}$, this provides a pleated hypersurface obtained by successively bending totally geodesic hyperplanes along lifts of $L$. Crucially, we choose the covers so that the distinct lifts of the bending locus are separated by distance at least a constant depending on $\theta$. The details are stated in Lemma~\ref{r vs theta for the bending locus} and Proposition~\ref{prop_reglue}.

This separation condition ensures that the convex hull of the resulting pleated hypersurface is simply connected and remains at uniformly bounded distance from the bent plane. By Proposition~\ref{prop_pi1 injectivity}, we obtain a $\pi_1$-injective hypersurface in $M$. 

Finally, we smooth the hypersurface in a small neighborhood of the bending loci. This smoothing does not alter the fundamental group and preserves $\pi_1$-injectivity (Corollary~\ref{cor_pi1 injectivity}). 
\begin{lem}\label{lem_angle density}
	Let $M$ be a closed arithmetic hyperbolic manifold of type I. Given any open interval $I\subset (0,\pi)$, there are infinitely many pairs of totally geodesic hypersurfaces $\Sigma_1, \Sigma_2$ which transversally intersect somewhere at an angle $\theta\in I$. Here we can take $\Sigma_1=\Sigma_2$. 
\end{lem}
\begin{proof}
	Choose a point $x\in M$ and a pair of hyperplanes $\Pi_1,\Pi_2$ in $T_xM$ that intersect transversally and form an angle $\alpha\in I$. Since a sequence of totally geodesic hypersurfaces $\Sigma_j$ equidistribute, the Grassmann bundles $\gnmo \Sigma_j$ converge in the Hausdorff metric to $\gnmo M$. We will approximate the angle by two intersecting hypersurfaces in this sequence. 
	
	There exist small neighborhoods $\mathcal{U}_i$ of $(x,\Pi_i)$ in $\gnmo M$ ($i=1,2$) and indices $j,k$, such that $x_j\in \Sigma_j$ and  $x_k\in \Sigma_k$ with $(x_j,T_{x_j}\Sigma_j)\in \mathcal{U}_1$ and $(x_k,T_{x_j}\Sigma_k)\in \mathcal{U}_2$. Moreover, we assume that $x_j,x_k\in B_\epsilon(x)$. When $\epsilon$ is sufficiently small, $\Sigma_j\cap B_\epsilon(x)$ and $\Sigma_k\cap B_\epsilon(x)$ can be viewed as smooth graphs over $\Pi_1$ and $\Pi_2$, respectively, with small $C^1$ norms. Transversality and the angle between two hyperplanes are open conditions in the $C^1$ topology. As a consequence, given a fixed $\delta>0$ arbitrarily small, we can find $\Sigma_j$ and $\Sigma_k$ which intersect in a lower-dimensional subspace near $x$, and the angle satisfies $\theta\in (\alpha-\delta, \alpha+\delta)\subset I$. 
	
	Since we can find infinitely many indices $j$ approximating $\Pi_1$ and infinitely many $k$ approximating $\Pi_2$, we can form infinitely many distinct pairs $(\Sigma_j,\Sigma_k)$ with angles in the given interval $I$. 
	
	Let $\mu_j$ be the probability measure induced by $\Sigma_j$ on $\gnmo M$. By \cite{msErgodicInvariantMeasure}, $\mu_j$ converge to the Liouville measure. Thus $$\mu_j(\mathcal{U}_1) \mu_j(\mathcal{U}_2) \rightarrow \mu_{\mathcal{L}}(\mathcal{U}_1)\mu_{\mathcal{L}}(\mathcal{U}_2)>0.$$ Thus all sufficiently large totally geodesic hypersurface must have self-intersections with angle closed to $\angle \Pi_1, \Pi_2$. 
\end{proof}
Later we lift $\Sigma_1,\Sigma_2$ to some finite-covers so they are embedded and have controlled intersections. If $\Sigma_1$ and $\Sigma_2$ are chosen from distinct commensurability classes, then by \cite{gpsNonarithmeticHyperbolic}, the hypersurface obtained from cut-and-paste will have non-arithmetic hyperbolic path metric. If we take $\Sigma_1=\Sigma_2$, we obtain an arithmetic $n$-manifold. 

\begin{lem}\label{r vs theta for the bending locus}
	Let $Q_1, Q_2\subset\mathbb{H}^{n+1}$ be totally geodesic half-hyperplanes intersecting transversally along an $(n-1)$-dimensional totally geodesic subspace $L=Q_1\cap Q_2$, with dihedral angle $\theta\in (0,\pi)$ along $L$. Fix a  sufficiently small $\epsilon>0$ and set $B=M_\epsilon(Q_1\cup Q_2)$. For $p\in\partial M_\epsilon(Q_i)$, let $P_p$ denote the unique totally geodesic hyperplane with $P_p\cap M_\epsilon(Q_i) = \{p\}$. 
	Then with \begin{equation}\label{equ_min_radius}
		r(\theta) = \arccosh \left(\cosh(\epsilon)\sqrt{1+\cot^2\left(\frac{\theta}{2}\right)\coth^2\epsilon}\right),
	\end{equation}
	the following holds: 
	\begin{center}
		if $\dist(p, L)>r(\theta)$, then $P_p\cap B = \{p\}$.
	\end{center}
\end{lem}

\begin{proof}
	Let $\Sigma$ be the totally geodesic 2-plane that is orthogonal to $L$ and contains $p$.  Then it suffices to consider the following objects in $\Sigma\cong\mathbb{H}^2$:
	\begin{itemize}
		\item $\alpha_i:=Q_i\cap\Sigma$ are geodesic segments intersecting with the same angle $\theta$ at $o:=\alpha_1\cap\alpha_2 = L\cap\Sigma$,
		\item $B\cap\Sigma = M_\epsilon(\alpha_1\cup \alpha_2)$,
		\item $P_p\cap \Sigma$ is the supporting geodesic to $B\cap \Sigma$ at $p$ on the exterior side.
	\end{itemize}
	Note that $P_p\cap B=\{p\}$  if and only if $(P_p\cap \Sigma)\cap (B\cap \Sigma)=\{p\}$, so it remains to prove the latter. 
	
	We assume $p\in \partial_\epsilon(\alpha_1)$ and consider two components of $\partial M_\epsilon (\alpha_1)\setminus B_\epsilon(o)$ separately. Let $\partial^+M_\epsilon(\alpha_1)$ be the component lying in the convex component of $\Sigma\setminus(\alpha_1\cup \alpha_2)$, and denote the other by $\partial^-M_\epsilon(\alpha_1)$.
	
	Suppose $p\in \partial^-M_\epsilon(\alpha_1)$. The half-line $\alpha_1$ extends to a geodesic $\overline{\alpha}_1$ in $\Sigma$, and let $\Sigma^+_1$ be the half-plane bounded by $\overline{\alpha}_1$ that contains $p$. When $r>\epsilon$, $M_\epsilon(\overline{\alpha}_1)$ can only intersect $P_p\cap \Sigma$ at $p$, so we must have $(P_p\cap\Sigma)\subset \Sigma^+_1$. Since $M_\epsilon(\alpha_2)\cap \Sigma^+_1\subset B_\epsilon(o)$ and $\dist(p,o)>r>\epsilon$, $M_\epsilon(\alpha_2)$ never touches $P_p\cap\Sigma$, which leads to the result. 
	
	Suppose $p\in \partial^+M_\epsilon(\alpha_1)$. Let $\beta\subset \Sigma_1^+$ be a geodesic ray from $o$ bisecting the angle $\theta$, and let $\{c(t)\}$ be a sequence of geodesics intersecting $\beta$ orthogonally with distance $t$ from $o$. There exists $t_0>0$ such that $c(t_0)$ is tangent to $B\cap\Sigma$ at two distinct points, lying respectively on $\partial_\epsilon(\alpha_1)$ and $\partial_\epsilon(\alpha_2)$.
	
	Now we calculate $r(\theta)$ in \eqref{equ_min_radius}. Let $c=c(t_0)\cap \beta$, and let $pq$ be 
    the common orthogonal between $c(t)$ and $\alpha_1$, where $p\in c(t), q\in \alpha_1$. 
     We then have a Lambert quadrilateral $ocpq$, where the angles at the vertices $c,p,q$ are right angles, and $\angle coq=\frac\theta2$ is an acute angle. We will express the diagonal $r=|op|$ in terms of $\theta$. Since $|pq|=\epsilon$,  by [\citenum{buser2010geometry}, 2.3.1 (vi)] we have the following
    \begin{equation}\label{equ_Lambert quadrilateral}
		\cot\frac\theta2 = \sinh(|oq|)\tanh\epsilon.
	\end{equation} 
	Moreover, [\citenum{buser2010geometry}, 2.2.2 (i)] applies to the hyperbolic right triangle $\triangle opq$ and implies that \begin{equation}\label{equ_Pythagoras_right_triangle}
		\cosh r = \cosh\epsilon\cosh(|oq|).
	\end{equation}
	Then \eqref{equ_min_radius} follows from \eqref{equ_Lambert quadrilateral} and \eqref{equ_Pythagoras_right_triangle}. Therefore, $r(\theta)$ decreases as $\theta$ increases for $\theta\in (0,\pi)$.
\end{proof}
The following is a consequence of [\citenum{buser2010geometry}, 2.3.1 (iii)]. 
\begin{cor}[Estimate of width]\label{width of convex hull}
    Let $Q=\hat{\Sigma}_1\cup \hat{\Sigma}_2$ and $w$ be the width of the convex hull $\ch(Q)$. Then $w$ satisfies $$
        \cosh w = \frac{\cosh\epsilon}{\sin\frac{\theta}{2}}.
    $$
    Hence $w$ tends to $0$ as $\theta\to\pi$ and $\epsilon\to 0$.
\end{cor}

\begin{prop}\label{prop_reglue}
	Successively bending hyperplanes with width $\geq r(\theta)$ and angle $\geq \theta$ makes a pleated plane $Q$ whose convex hull $\ch(Q)$ is simply connected and has bounded Hausdorff distance to $Q$. 
\end{prop}
\begin{proof}
	Since we inter-breed along an embedded codimension-$2$ \totallygeodesic submanifold $F$ of two hypersurfaces, in the universal cover $\hs$ we have a family of planes $\{P_i\}$ such that there are no triple-intersections between $3$ planes. 
	Let $Q$ be a pleated plane formed from half-planes of $\{P_i\}$ whose bending loci consist of \totallygeodesic codimension-$2$ subspaces of $\hs$ such that the subspaces have pairwise distance $\geq 2r(\theta)+\eta$ for $\eta>0$. 
	
	Consider the $\epsilon$-neighborhood $M_{\epsilon}(Q)$ of $Q$. We show that the convex hull $\ch(Q)$ can be constructed by adding a \emph{disjoint} union of saw-wing regions to $M_{\epsilon}(Q)$. By \Cref{r vs theta for the bending locus}, if $p\in M_{\epsilon}(Q)$ has distance $\geq r(\theta)$ from all the bending loci, there exists a unique tangent totally geodesic plane $P_p \cap M_{\epsilon}(Q)=p$. If $p$ has distance $\leq r(\theta)$ to some of the bending locus, Masters in proving [\citenum{mjCountingImmersedSurfaces}, Lemma 4.2] has shown that adding a saw-wing part to the union of the $\epsilon$-neighborhood of a union of two intersecting half-planes results in a convex set (see also [\citenum{mjCountingImmersedSurfaces}, Figure 3]), by proving the key property that each point $p$ on the cap of a saw-wing region has a tangent hyperplane whose intersection with the saw-wing is $\{p\}$. 
	
	The saw-wing regions added near each locus depend only on $\theta$ and $r$ and have bounded distance to the bent planes. By adding one saw-wing region to $M_{\epsilon}(Q)$ corresponding to each bending locus, we make a set $B_s$ where each point $\{p\}$ on its boundary has a tangent hyperplane $P$ such that $P\cap B=\{p\}.$ Thus $B_s$ is convex. 
\end{proof}

Next, we complete the proof of $\pi_1$-injectivity following the strategy of Lemma 4.3 in \cite{mjCountingImmersedSurfaces}.
\begin{prop}\label{prop_pi1 injectivity}
    Let $X$ be a metrically complete hyperbolic $(n+1)$-manifold with convex boundary. Then any locally isometric immersion $f\colon  X\to M$ induces an injective homomorphism $f_*\colon \pi_1(X)\to \pi_1(M) $. 

    In particular, $$X:=\mathrm{Stab}(\ch(Q))\backslash\ch(Q)=\mathrm{Stab}(Q)\backslash\ch(Q)$$
    is a metrically complete hyperbolic $(n+1)$-manifold with convex boundary.
    Thus $\mathrm{Stab}(Q)\to \pi_1(M)$ is injective and $S_0:=\mathrm{Stab}(Q)\backslash Q,$
    the quotient of the pleated hypersurface $Q$ by its stabilizer, is a $\pi_1$-injective pleated hypersurface in $M$. 
\end{prop}
\begin{proof}
    Since $\partial X$ is convex, the universal cover of $X$ is isometrically embedded as a convex domain $C$ in $\mathbb{H}^{n+1}$ (see [\citenum{cegNotesNotesThur}, I.1.4.2]). Let $\Tilde{f}\colon  C\to \mathbb{H}^{n+1}$ be a $\pi_1(X)$-equivariant local isometry lifted by $f$. Since $C$ is a convex subset of $\mathbb{H}^{n+1}$, it is $\cat(-1)$ and therefore there exist a unique geodesics between any two points. This implies that $\Tilde{f}$ must be distance non-increasing, and because the target is simply connected, it must be injective. 

    For any $g \in \ker(f_*)$ and any $x\in C$, we have $\Tilde{f}(g \cdot x)=f_*(g )\cdot \Tilde{f}(x)=\Tilde{f}(x)$. By injectivity of $\Tilde{f}$, it follows that $g \cdot x=x$. As the $\pi_1(X)$ action on $C$ is free, the only deck transformation with a fixed point is the identity. Therefore, $g =1$ and $f_*$ is injective. 
\end{proof}

Since isotopy preserves the induced map on $\pi_1$, we derive the following corollary.
\begin{cor}[$\pi_1$-injectivity]\label{cor_pi1 injectivity}
    Let $S\subset M$ be the hypersurface obtained by smoothing the hypersurface $S_0$ in the previous proposition, where the smoothing is achieved by a small isotopy supported in a tubular neighborhood of the bending locus. Then $S$ is $\pi_1$-injective in $M$.
\end{cor}
The $\pi_1$-injectivity of the constructed hypersurfaces $S$ can also be deduced from the \ag property of $S$ in \Cref{prop_asymp_geodesic} and [\citenum{lcSmallCurvatureSurface}, Proposition 5.1] that \nege  hypersurfaces are \poin. However, the proof above also works for hypersurfaces that are not nearly geodesic. 

\subsubsection{Geometric properties of the resulting hypersurface}
To conclude this section, we present several additional properties of the $\pi_1$-injective hypersurface $S\subset M$ constructed in the previous subsection. We prove that $S$ is an immersion of closed hyperbolic $(n+1)$-manifold that is not homotopic to a totally geodesic hypersurface.

Yi Liu points out such surgered hypersurfaces are hyperbolizable. 
\begin{prop}\label{supporting real hyperbolic}
	Our constructions, which come from cutting and regluing two embedded totally geodesic hypersurfaces along one of their embedded totally geodesic hypersurface, are immersions of closed hyperbolic manifold. 
\end{prop}
\begin{proof}
	Take two closed $n$-dimensional hyperbolic manifolds $M_1, M_2$, with an embedded nonseparating totally geodesic hypersurfaces $L_1\hookrightarrow M_1$, $L_2\hookrightarrow M_2$ such that $L_1$ is isometric to $L_2$. 
	We glue $M_1-L_1$ and $M_2-L_2$ along their isometric boundary. 
	Note that gluing two hyperbolic manifolds along totally geodesic boundary via isometry makes an authentic hyperbolic manifold. 
\end{proof}

\begin{prop}\label{non totally geodesic}
    The hypersurface $S$ is neither totally geodesic nor homotopic to a totally geodesic hypersurface of $M$.
\end{prop}
\begin{proof}
    We first prove that $S$ itself is non totally geodesic, which is equivalent to showing that the second fundamental form $\sff_S$ is not identically zero. Assume for contradiction that $\sff_{\hat{S}} \equiv 0$ on $\hat{S}\cap U$, where $U\subset\mathbb{H}^{n+1}$ is a tubular neighborhood of the bending locus that contains the smoothing region in the above construction. Then $\hat{S}\cap U$ is totally geodesic, hence it is contained in a unique totally geodesic hyperplane $P\subset\mathbb{H}^{n+1}$. Since $P$ coincides with $Q_1$ and $Q_2$ in two open subsets of $U\setminus L$, we must have $Q_1\subset P$ and $Q_2\subset P$, and therefore $Q_1=Q_2$. This contradicts $\theta\subset (0,\pi)$.

We now prove that $S$ is not homotopic to smoothly immersed totally geodesic hypersurfaces. Recall that $S_0$ is the hypersurface constructed in Proposition~\ref{prop_pi1 injectivity} by cutting and regluing two totally geodesic pieces. We have that $S_0$ is homotopic to $S$, and $S_0$ is pleated whose volume satisfies
$$
\vol(S_0) = v_n  \Vert S_0 \Vert,
$$
where $\Vert S_0 \Vert$ denotes the simplicial volume that is homotopy invariant, and $v_n $ is the maximal volume of an ideal $n$-simplex in $\mathbb H^n $. 

By \Cref{totally geodesic area-minimizing}, a totally geodesic immersion minimizes area within its homotopy class. The construction replaces a neighborhood of the locus by circular smoothing. The circular part is convex. By [\citenum{bhMetricSpaceNonpositive}, II.2.5], the nearest point projection to a convex subset is $1$-Lipschitz. Moreover, in hyperbolic space, the distance function between two geodesics is strictly convex. There is a nearest point projection from the neighborhood to the circular part, which is strictly contracting. Thus the area decreases after the surgery. 
This constructs a smooth hypersurface $S_0'\subset M$ homotopic to $S_0$ with $\vol(S_0') < \vol(S_0)$. If $S_0'$ were homotopic to a smooth totally geodesic hypersurface $S_0''$, then it would produce a contradiction to Theorem~\ref{totally geodesic area-minimizing}:
$$
\vol(S_0') < \vol(S_0) = v_n  \Vert S_0 \Vert = \vol(S_0'').
$$
\end{proof}
By \Cref{nearly geodesic hypersurfaces}, $S$ is negatively curved. Moreover, as a consequence of \Cref{nearly geodesic hypersurfaces} and \Cref{almost geo of arithmetic are virtually embedded}, the sequence of hypersurfaces constructed is eventually virtually embedded and virtually homologically injects. 

\subsection{The equidistribution}\label{subsection_equidistribution}
In this subsection, we begin with a sequence of totally geodesic hypersurfaces $\Sigma_i$ that equidistribute in $M$. Applying the construction of Subsection~\ref{subsection_construction of pi1-injective hypersurface} to each $\Sigma_i$, we obtain a sequence of $\pi_1$-injective hypersurfaces $S_i$. We then establish the equidistribution properties of the sequence $\{S_i\}$.

Let $M$ be a closed arithmetic hyperbolic $(n+1)$-manifold and let 
$\Sigma_i\subset M$ be a sequence of immersed totally geodesic hypersurfaces whose
associated probability Radon measures $\mu_{\Sigma_i}$ on the Grassmannian bundle 
$\gm M$ equidistribute.  This means that $\mu_{\Sigma_i}$ converges to the Liouville measure $\mu_{\mathcal{L}}$ in the weak-* topology.
For each $i$, let $L_{i}^0$ be a closed totally geodesic 
submanifold of codimension $2$ arising from a transverse self-intersection 
of $\Sigma_i$, that is, $\Sigma_i$ lifts to embedded finite covers $\hat{\Sigma}_{i,1}^0$ and $\hat{\Sigma}_{i,2}^0$ such that $\hat{\Sigma}_{i,1}^0\cap \hat{\Sigma}_{i,2}^0 = L_i^0$.
After passing to suitable finite covers $\hat{\Sigma}_{i,1}, \hat{\Sigma}_{i,2}$ of $\hat{\Sigma}_{i,1}^0, \hat{\Sigma}_{i,2}^0$, the associated cover $L_i$ of $L_i^0$ is nonseparating. 

Following the above procedures, one can construct a sequence of smooth immersed $\pi_1$-injective hypersurfaces $S_i\subset M$. The lift $\hat{S}_i$ is obtained by cutting and regluing along $L_i$ and smoothing inside a tubular neighborhood $M_{r}(L_i)$.
By Lemma~\ref{lem_angle density}, we may assume that the angle between $\hat{\Sigma}_{i,1}$ and $\hat{\Sigma}_{i,2}$ converges to $\pi$ as $i\to\infty$. It then follows from Lemma~\ref{r vs theta for the bending locus} that the radius $r>r_i$ is a uniform constant for sufficiently large $i$. 

In the lemma below, we prove the equidistribution of the probability Radon measures $\mu_{S_i}$ associated with $S_i$ on $\gm M$.

\begin{prop}[Equidistribution property]\label{prop_equidistribution}
Let $B_i:= \hat{S}_i\cap M_r(L_i)$ be the smoothing region. Then $$ 
    \frac{\vol(B_i)}{\vol(\hat{S}_i)} \to 0, \quad i\to \infty.
$$ 
\end{prop}
This implies that the measure induced by $B_i$ is negligible in $\mu_{S_i}$. Thus $\mu_{S_i}$ shares the same limit with $\mu_{\Sigma_i}$. Therefore, $$ 
    \mu_{S_i} \overset{\ast}{\rightharpoonup} \mu_{\mathcal{L}}, \quad i\to\infty.
$$  

\begin{proof}
Given the choices of $\hat{\Sigma}_{i,1}^0, \hat{\Sigma}_{i,2}^0, L_i^0$, we now discuss the precise construction of $\hat{\Sigma}_{i,1}$, $\hat{\Sigma}_{i,2}$, $L_i$.

By subgroup separability for geometrically finite subgroups of arithmetic 
lattices, there exists a finite index subgroup 
$H_i< \pi_1(\hat{\Sigma}_{i,j}^0)$ such that the inclusion 
$\pi_1(L_i^0)\cap \pi_1(\hat{\Sigma}_{i,j}^0) \hookrightarrow \pi_1(\hat{\Sigma}_{i,j}^0)$ lifts injectively to the cover 
$\hat{\Sigma}_{i,j}^1\to \hat{\Sigma}_{i,j}^0$ corresponding to $H_i$, where $j=1,2$. 
Equivalently, a lift $L_i^1\subset \hat{\Sigma}_{i,j}^1$ of $L_i^0$ is embedded and nonseparating. Furthermore, this ensures that the degree of the covering restricted to $L_i^0$ is $$ 
   c_i:=[\pi_1(L_i^0): H_i\cap \pi_1(L_i^0)]< \infty,
$$ 
where $c_i$ depends only on the triple $(\hat{\Sigma}_{i,1}^0, \hat{\Sigma}_{i,2}^0, L_i^0)$.

Next, we choose a further finite index subgroup $G_i< H_i$ of large index.  Let $$ 
    d_i: = [\pi_1(\hat{\Sigma}_{i,j}^0):G_i] 
$$ 
which will be determined later. 
The lifted hypersurface $\hat{\Sigma}_{i,j}$ in the cover corresponding to $G_i$ satisfies
$$ 
    \vol(\hat{\Sigma}_{i,j}) = d_i \,\vol(\hat{\Sigma}_{i,j}^0).
$$ 
The lifted submanifold $L_i\subset \hat{\Sigma}_{i,j}$ satisfies
$$ 
    \vol(L_i) \leq c_i\, \vol(L_i^0).
$$ 

In tubular coordinates, $M_r(L_i) \cong L_i \times D^2_{r}$, and $\hat{S}_i$ is obtained from $\hat{\Sigma}_{i,j}$ by performing the cut-and-paste construction along $L_i$ and smoothing only inside $M_r(L_i)$.
By the construction of the smoothing region, $B_i=\hat{S}_i \cap M_r(L_i)$
is a uniformly Lipschitz graph over $L_i$, contained in $L_i \times G _i$, where $G _i\subset D^2_r$ is a smooth curve whose length admits an upper bound depending only on $r$. Hence,
$$ 
    \vol(B_i) = \vol\left(\hat{S}_i\cap M_r(L_i)\right)
\leq c(r)\,\vol(L_i),
$$ 
for a constant $c(r)$ depending only on $r$.

Combining the above estimates yields
$$ 
    \vol(B_i)\leq c(r) \,c_i\,\vol(L_i^0).
$$ 
On the other hand,
$$ 
    \vol(\hat{S}_i)> \vol(\hat{\Sigma}_{i,j})= d_i\, \vol(\hat{\Sigma}_{i,j}^0).
$$ 
Thus
$$ 
    \frac{\vol(B_i)}{\vol(\hat{S}_i)}\leq \frac{c(r)\, c_i \, \vol(L_i^0)}{d_i\,\vol(\hat{\Sigma}_{i,j}^0)}.
$$ 

Since $d_i$ can be chosen arbitrarily large, while 
$c_i$, $\vol(L_i^0)$ and $\vol(\hat{\Sigma}_{i,j}^0)$ are predetermined, 
we may arrange $\vol(B_i)/\vol(\hat{S}_i) < \frac{1}{i}$. Hence $\vol(B_i)/\vol(\hat{S}_i)\to 0$, completing the proof.
\end{proof}

\subsection{The asymptotically geodesic property}\label{subsection_asymptotically geodesic}
In this subsection, we describe a more precise smoothing procedure near the codimension-$2$ intersection locus. This refined smoothing ensures that the sequence $\{S_i\}$ is asymptotically geodesic. 

Fix $i$ and let $M_{r}(L_i)$ be the tubular neighborhood of $L_i$, where $r$ is a uniform constant to be determined later.
The normal exponential map provides a diffeomorphism $$\exp^\perp\colon  L_i\times D^2_{r} \xrightarrow{\cong} 
M_{r}(L_i),$$ where $D^2_{r}\subset\mathbb{H}^2$ denotes the geodesic disk of radius 
$r$ in the normal $2$-plane.
In these tubular coordinates, the totally geodesic hypersurfaces 
$\hat \Sigma_{i,1}$ and $\hat \Sigma_{i,2}$ correspond to two totally geodesic
rays in each normal $2$-plane based at the origin.  
More precisely, for each $x\in L_i$, the intersection $(\hat \Sigma_{i,1}\cup\hat \Sigma_{i,2})\cap 
(\{x\}\times D^2_{r})$ is the union of two geodesic rays in $D^2_{r}$ meeting at an angle 
$\theta_i\in(0,\pi)$ at the origin.  Assume that $\theta_i\to\pi$ as $i\to\infty$. Furthermore, let $$r_i:=r(\theta_i)$$ be the constant chosen in Lemma~\ref{r vs theta for the bending locus}. Since $r(\theta)$ decreases by $\theta$, we may choose a radius $r$ of the tubular neighborhood so that $r>r_i$ for any sufficiently large $i$.

After cutting and regluing along $L_i$ as in the previous section, we obtain a $C^0$ hypersurface $\hat{S}_i^{(0)}$ which coincides with $\hat{\Sigma}_{i,1}\cup \hat{\Sigma}_{i,2}$ outside $M_{r}(L_i)$ and whose intersection with each normal $2$-plane is a union of two rays 
meeting with angle $\theta_i$.

We now smooth the corner of $\hat{S}_i^{(0)}$ inside the tubular 
neighborhood. Fix $x\in L_i$ and consider the geodesic rays in
$\{x\}\times D^2_{r}$. 
On each of the rays we mark the point at distance $r$ from the origin. There exists a unique Euclidean circle $S^1_{i}$ whose center lies on the angle bisector and which is tangent to both rays at these marked points. 

The radius of $S^1_{i}$ is 
$$
R_i = r\tan\left(\frac{\theta_i}{2}\right).$$
Since $\theta_i\to\pi$, we have $\tan\left(\frac{\theta_i}{2}\right)\to\infty$, and hence $\lim_{i\to\infty} R_i=\infty$.

We now construct a smooth curve $G _{i,x}\subset D^2_{r}$ which agrees with the union of the two geodesic rays outside $D_r$ and coincides inside $D_r$ with $S^1_{i}\cap D_r$. In the Euclidean metric the curvature of this arc is $1/R_i$, and by the uniform $C^2$-closeness of the hyperbolic metric to the Euclidean metric on $D^2_{r}$, we obtain a constant $C$ that is independent of $i$, such that the geodesic curvature of $G _{i,x}$ in the hyperbolic metric satisfies
\begin{equation}\label{equ_geodesic_curvature_bound}
    k_g(G _{i,x}) \leq \frac{C}{R_i}.
\end{equation}

By construction the map $x\mapsto G _{i,x}$ can be chosen to 
depend smoothly on $x\in L_i$, so that the union
$$ 
    B_i:= \bigcup_{x\in L_i} \{x\}\times G _{i,x}\subset L_i\times D^2_{r}
$$ 
is a smooth embedded hypersurface in the tubular neighborhood.  
Outside $B_i$ we keep $\hat{S}_i^{(0)}$ unchanged.  
This yields a smooth embedded hypersurface $\hat{S}_i$ which agrees with $\hat{\Sigma}_{i,1}\cup\hat{\Sigma}_{i,2}$ outside 
$M_{r}(L_i)$ and coincides with $B_i$ inside the tubular 
neighborhood.

We consider the second fundamental form on the smoothing region $B_i$. Along $L_i$ the hypersurface is totally geodesic in the tangential directions. Thus the only nonzero principal curvature comes from the normal $2$-plane direction, and is controlled by the geodesic curvature $k_g(G _{i,x})$.

\begin{prop}[Asymptotically geodesic property]\label{prop_asymp_geodesic}
Let $S_i\subset M$ be the smooth hypersurfaces constructed
above by surgery on the totally geodesic hypersurfaces 
$\hat \Sigma_{i,1},\hat \Sigma_{i,2}$ intersecting along $L_i$.  
Let $\theta_i\in(0,\pi)$ denote the intersection angle of 
$\hat \Sigma_{i,1}$ and $\hat \Sigma_{i,2}$ with $\theta_i\to\pi$.
Then the second fundamental form of $S_i$ satisfies
$$ 
    \sup_{S_i} |\sff_{S_i}| \to 0
\quad\text{as } i\to\infty.
$$ 
Thus $S_i$ are asymptotically geodesic.
\end{prop}

\begin{proof}
As observed above, at each point of $B_i$, all but one principal curvature vanish, and the remaining nonzero principal curvature is comparable to the geodesic 
curvature $k_g(G _{i,x})$ in the normal $2$-plane (up to a uniform constant depending on $r$).  
Thus by \eqref{equ_geodesic_curvature_bound}, we derive that $$ 
    |\sff_{S_i}| \lesssim k_g(G _{i,x})\leq \frac{C}{R_i}
\to 0.
$$ 

Furthermore, all derivatives of $\gamma_{i,x}$ of order $\geq 2$ are uniformly bounded, so $\gamma_{i,x}$ converges smoothly to a geodesic arc in $D^2_r$ as $\theta$ tends to $\pi$, uniformly in $x$. It then follows that the exponential map has uniformly bounded derivatives on $D_r^2$. Hence $\hat{S}_i\cap B_i$ converges smoothly on compact subsets in $M_r(L_i)$. Outside this tubular neighborhood, $\hat{S}_i$ coincides identically with the totally geodesic hypersurfaces $\hat{\Sigma}_{i,1}\cup \hat{\Sigma}_{i,2}$, which also ensures the smooth convergence. Therefore, $\hat{S}_i$  is  asymptotically geodesic. 
\end{proof}

Theorem~\ref{thm_main_of_sec4} follows by combining Corollary~\ref{cor_pi1 injectivity} and Propositions~\ref{prop_equidistribution} and~\ref{prop_asymp_geodesic}. By choosing $\Sigma_{i,1}=\Sigma_{i,2}$ in the construction, one obtains arithmetic $n$-manifolds $S_i$. By choosing $\Sigma_{i,1}$ and $\Sigma_{i,2}$ from distinct commensurability classes of totally geodesic hypersurfaces, one obtains non-arithmetic examples by~\cite{gpsNonarithmeticHyperbolic}. Moreover, the proof yields infinitely many examples for which $\pi_1(S_i)$ is arithmetic, and  infinitely many examples for which $\pi_1(S_i)$ is non-arithmetic. This proves the final claim of \Cref{arithmetic of type I contains nice hypersurfaces}.

\begin{rem}
	From the proof it is clear that \Cref{arithmetic of type I contains nice hypersurfaces} also holds for finite-volume arithmetic hyperbolic manifolds which contain a closed totally geodesic hypersurfaces. However, for dimensions $n \geq 5$, no non-compact arithmetic hyperbolic $n$-manifolds contain a co-dimension-$1$ closed totally geodesic hypersurface. The fundamental groups of such manifolds are defined by quadratic forms of signature $(n,1)$ over $\Q$. Meyer's theorem says that all such forms of signature $(n,1)$ over $\Q$ are isotropic whenever $n\geq 4$. Since a co-dimension-$1$ totally geodesic hypersurface comes from a form over $\Q$ of signature $(n-1,1)$, whenever $n-1\geq 4$, the form is isotropic, and so the hypersurface is noncompact.
\end{rem}
A natural question which clarifies the construction is the necessity of making $\theta$ the bending angle tend to $\pi$ in order to construct a sequence of \ag hypersurfaces. Both constructions in [\citenum{mjCountingImmersedSurfaces}, Section 4] and [\citenum{kmCountingEssentialSurfaces}, Section 4] in a closed hyperbolic $3$-manifold $M$ rely on cut-and-paste on a pair of (nearly) geodesic surfaces which are intersecting with uniformly bounded angle. First of all, \cite{alAsymptoticallyGeodesicsurfaces,hxhNearlyGeodesicFilling} prove that \ag surfaces in hyperbolic $3$-manifold are asymptotically dense and by \Cref{ag implies eventually strongly filling} are linked with every geodesic. In [\citenum{mjCountingImmersedSurfaces}, Section 4] Master takes a single self-intersecting totally geodesic surface $S$, and apply surgery to increasingly higher degree covers of $S$. If $S$ is disjoint from a closed geodesic $\gamma$, then Master's construction will produce a sequence $S_n$ of surfaces that is also homotopically disjoint from $\gamma$, and thus cannot be strongly filling or \agn. 

If we let $p_n\in M$ be a convergence sequence of points in the bending locus of $S_n$ with bending angle $\geq \theta$, then in [\citenum{mjCountingImmersedSurfaces}, Section 4] the lifts of the $R_n$-neighborhood of $p_n$ to the universal cover \emph{limit} to two intersecting totally geodesic half-planes $Q_1\cup Q_2$ with uniformly bounded $\theta$. If $S_n$ were \agn, the convex core should have width tending to $0$, and any bent geodesic $\beta$ on the limiting pleated plane $Q_1\cup Q_2$ should be arbitrarily close to the actual geodesic $\beta'$ with endpoints $\beta(\pm \infty)$. However with the uniform bent angle results in a uniform gap between $\beta$ and $\beta'$ and a uniform lower bound of the width of the convex cores by \Cref{width of convex hull}.

\subsection{Proof of \Cref{non superrigid}}\label{proof of 1.2}
\begin{proof}
	Let $\Gamma$ be a cocompact arithmetic lattice of $\SOPO$ of type I (with $n\geq 2$). Let $H$ be a cocompact lattice of $\SONO$ arising from the constructions above, so that $H$ may be arithmetic or non-arithmetic. Let $\iota_*\colon  H\to \Gamma$ be an injective homomorphism corresponding to a non totally geodesic immersion $S\hookrightarrow M$, where $H=\pi_1(S), \Gamma=\pi_1(M)$. Denote the inclusion maps as $\iota_H\colon  H \to \SONO$ and $\iota_\Gamma\colon \Gamma \to \SOPO$. 
Suppose there is a Lie group homomorphism $\alpha\colon  \SONO\to \SOPO$ that extends $\iota_*\colon  H\to \Gamma$. Then the following diagram commutes: 
$$
\begin{tikzcd}
	\SONO \arrow[r,"\alpha"]  &
	\SOPO  \\
	H \arrow[u,"\iota_H"'] \arrow[r,"\iota_*"'] & \Gamma\arrow[u,"\iota_\Gamma"].
\end{tikzcd}
$$
\begin{lem}\label{lie group homomorphism from SONO to SOPO}
The extension $\alpha$ is a Lie group embedding. 
\end{lem}
\begin{proof}
	A Lie group homomorphism $\alpha$ between Lie groups $G$ and $G'$ induces a Lie algebra homomorphism $d\alpha$. Moreover, when $G$ and $G'$ are connected, $\alpha$ is determined by $d\alpha$ (see, e.g., [\citenum{kaLieGroupBeyond}, 1.10]). 
The Lie algebras $\mathfrak{so}(n,1)$ are simple ([\citenum{kaLieGroupBeyond}, Theorem 6.105]). Thus any closed normal subgroup of $SO(n,1)^\circ$ is either trivial, finite central, or all of $SO(n,1)^\circ$. If \(n\) is even, \(-I \) is not in \(SO(n,1)\).
If \(n\) is odd, $-I$ is in $ SO(n,1)$, but not in $SO(n,1)^\circ$. Thus in both cases, $SO(n,1)^\circ$ is center-free. Since $\alpha \circ \iota_H(H)=\iota_*\circ \iota_\Gamma(H)$, $\ker\alpha$ must be $\{e\}$ and the differential
$
	d\alpha_e \colon  \mathfrak{so}(n,1) \to \mathfrak{so}(n+1,1)
$
	is a nonzero Lie algebra homomorphism. Since $\mathfrak{so}(n,1)$ is simple, any nonzero Lie algebra homomorphism is injective. Hence we obtain an embedding
	$\mathfrak{so}(n,1)\hookrightarrow\mathfrak{so}(n+1,1)$ and an embedding of $\SONO$ to $\SOPO$. 
\end{proof}
The rest of the proof follows from the idea of proof of [\citenum{mdArithmeticGroups}, (16.2.5) Proposition]. We provide some more details for readers unfamiliar with the connection between invariance under Cartan involution and totally geodesic submanifold. 

Let \(S = H\backslash G/K\), \(M = \Gamma \backslash G'/K'\) be symmetric spaces, where \(G,G'\) are center–free semisimple Lie groups with no compact factors, and \(K\subset G\), \(K'\subset G'\) are maximal compact subgroups. Let \(\mathfrak g = \mathrm{Lie}(G)\), \(\mathfrak g' = \mathrm{Lie}(G')\), and 
$
\theta\colon \mathfrak g\to\mathfrak g,
\theta'\colon \mathfrak g'\to\mathfrak g'
$
 Cartan involutions with fixed–point sets \(\mathfrak k = \mathrm{Lie}(K)\) and \(\mathfrak k' = \mathrm{Lie}(K')\). Then we have Cartan decompositions
$
\mathfrak g = \mathfrak k \oplus \mathfrak p,
\mathfrak g' = \mathfrak k' \oplus \mathfrak p',
$
with \(\theta|_{\mathfrak k} = +\mathrm{Id}\), \(\theta|_{\mathfrak p} = -\mathrm{Id}\), and similarly for \(\theta'\). Suppose there is a continuous embedding $\alpha\colon G \to G'$ and an injective homomorphism $\iota_*\colon H \to \Gamma$. 
\begin{lem}
	After possibly conjugating \(\alpha \) by an element of \(G'\), we have
	\[
	d\alpha _e\circ \theta
	=
	\theta'\circ d\alpha _e \tn{ or equivalently } d\alpha _e(\mathfrak k)\subset \mathfrak k',
	d\alpha _e(\mathfrak p)\subset \mathfrak p'.
	\]
\end{lem}
\begin{proof}
	By [\citenum{kaLieGroupBeyond}, Proposition 6.61.], any two maximal compact subgroups in a semisimple Lie group are conjugate. By conjugating \(\alpha \) by an element of \(G'\), we have \(\alpha (K)\subset K'\), i.e., $
	d\alpha _e(\mathfrak k)\subset \mathfrak k'.
	$
	
Let \(\tilde\theta'\) be a Cartan involution of \(\mathfrak g'\) whose fixed subgroup is conjugate to \(K'\). The restriction of \(\tilde\theta'\) to the subalgebra \(d\alpha _e(\mathfrak g)\) is itself a Cartan involution. On the domain \(\mathfrak g\), we have \(\theta\). By [\citenum{kaLieGroupBeyond}, Corollary 6.19], any two Cartan involutions on $\mathfrak g$ are conjugate via inner automorphisms.  Thus we can conjugate \(d\alpha _e(\mathfrak g)\) inside \(\mathfrak g'\) so that \(\tilde\theta'|_{d\alpha _e(\mathfrak g)}\) matches the pushforward of \(\theta\), since $\alpha(K)$ is a subset of $K'$. 
	Lifting this conjugation to \(G'\) gives the desired adjustment of \(\alpha \) by an inner automorphism of \(G'\).
	
	Since $G$ is connected, by exponentiating, we have $
	\theta'(\alpha (g)) = \alpha (\theta(g))$ for any $ g\in G $, which implies that $ \theta'(\alpha (G)) \subset \alpha (G).
$ 
	Since \(\theta'\) is an involution (and therefore bijective), 
	$
	\theta'(\alpha (G)) = \alpha (G).
	$
\end{proof}
	Here \(\theta'\)-invariant subgroups correspond to totally geodesic, symmetric subspaces of the target symmetric space \(G'/K'\), which is why \(\alpha \) induces a totally geodesic embedding of \(G/K\) into \(G'/K'=M\).
Let \(U \subset G'\) be a connected Lie subgroup such that
\(\theta'(U) = U\) and  \(K_U := U \cap K'\) is compact (it will be a maximal compact subgroup of \(U\)).

We claim that the natural map
\[
U/K_U \longrightarrow G'/K',\quad hK_U \mapsto hK'
\]
is a totally geodesic isometric immersion of Riemannian symmetric spaces.

The tangent space at the basepoint \(o = eK'\in X'\) can be identified with
$
T_o X' \cong \mathfrak p'.
$
By assumption, \(\mathfrak u = \mathrm{Lie}(U)\) is \(\theta'\)-invariant as a subalgebra of \(\mathfrak g'\), so
$
\mathfrak u = \mathfrak k_U \oplus \mathfrak p_U,
$
where
$
\mathfrak k_U := \mathfrak u \cap \mathfrak k',\quad
\mathfrak p_U := \mathfrak u \cap \mathfrak p'.
$
This is the Cartan decomposition of \(\mathfrak u\) corresponding to the symmetric pair \((U,K_U)\). The symmetric space of \(U\) with respect to this Cartan involution is
$
X_U = U/K_U,
$
whose tangent space at the basepoint \(o_U = eK_U\) is
$
T_{o_U} X_U \cong \mathfrak p_U.
$

Under the natural map
$
\iota \colon  U/K_U \longrightarrow G'/K',\quad hK_U \mapsto hK',
$
the derivative at the basepoint identifies
\[
d\iota_{o_U} \colon  \mathfrak p_U \hookrightarrow \mathfrak p'.
\]
Thus \(\iota\) is an isometric immersion at the basepoint (for the standard symmetric space metrics).

In \(X' = G'/K'\), geodesics through the basepoint \(o\) are exactly the curves
\[
\gamma_X(t) = \exp(tX)\cdot o,\quad X\in \mathfrak p'.
\]
Similarly, geodesics in \(X_U = U/K_U\) through its basepoint are
$
\gamma_Y^U(t) = \exp(tY)\cdot o_U, Y\in \mathfrak p_U.
$

Because \(\mathfrak p_U \subset \mathfrak p'\) and exponentials in \(U\) and \(G'\) coincide on \(\mathfrak u\),
\[
\exp_U(Y) = \exp_{G'}(Y)\quad\text{for }Y\in\mathfrak u,
\]
we have
$
\iota\bigl(\gamma_Y^U(t)\bigr)
= \iota\bigl(\exp(tY)\cdot o_U\bigr)
= \exp(tY)\cdot o,
$
and the right-hand side is a geodesic in \(X'\). Thus every geodesic in \(X_U\) through the basepoint is carried by \(\iota\) to a geodesic in \(X'\). 
This proves that \(\iota\) is an isometric immersion at all points, and it preserves geodesics.

By \Cref{supporting real hyperbolic}, each such $H$ is a cocompact lattice in $SO(n,1)^\circ$. By \Cref{non totally geodesic}, the corresponding immersion is not homotopic to any totally geodesic hypersurface of $M$. Hence $\iota_*$ cannot extend to a Lie group homomorphism.
Finally, by the construction earlier in this section, there are infinitely many arithmetic and infinitely many non-arithmetic cocompact lattices in $SO(n,1)^\circ$. This completes the proof of \Cref{non superrigid}.

\end{proof}
\subsection{Some questions}
\begin{q}
	Are there minimal hypersurfaces in the homotopic class of a \poi hypersurface in a hyperbolic manifold? Essentially, generalizing \cite{syExistIncomMinTopo,suMinimalImmersionClosedRiemann}\
to higher-dimensional hyperbolic manifolds. 
\end{q}

\begin{q}
	What is an upper bound on the growth rate of \poi hypersurfaces in a finite-volume hyperbolic manifold? \cite{kmCountingEssentialSurfaces} counts the number of homotopy classes of \poi surfaces in a closed hyperbolic $3$-manifold, by realizing every surface as a pleated surface (piecewise hyperbolic structure), and discretizing both the surfaces and the manifold as graphs. There is a clean description of thick graphs which are $\pi_1$-isomorphic to the surface. 
	Moreover, it is not clear whether there exist \poi immersions of Gromov-Thurston type negatively curved manifolds (which do not support hyperbolic structure) into hyperbolic manifolds. 
\end{q}

\section*{Acknowledgment}
 We are grateful for Ben Lowe, Alan W. Reid, and Wenyuan Yang for many helpful discussions. We are in debt to \vl \mkv for suggesting \Cref{ag implies asymptotically dense} and most surfaces Jeremy Kahn and he construct in \cite{kmImmersingAlmostGeodesic} should be filling. 
This work initiated in a discussion between Mark Hagen and XHH when both were attending the conference ``Geometric Group Theory and Related Topics" at Peking University. We would like to thank Mark Hagen and the organizers of the conference for many inspirations.  
XHH benefits significantly from the reading seminars on \cite{ukclosedminSurfaceHyperbolic,cmnCountingminimal,kmsGeometricallyTopologicallyRandomsurface} at Yau Mathematical Sciences Center at Tsinghua University, mentored by Yunhui Wu. He is truly grateful for the participants. This work is partially completed when XHH visited California Institute of Technology, and he thanks the institute for its hospitality and thank Antoine Song for many helpful discussions. 
XHH is partially supported by the start-up grants at Shanghai Institute of Mathematics and Interdisciplinary Sciences and  NSFC No. 12501084.

\bibliographystyle{alpha}
\bibliography{9-Reference.bib}

\end{document}